\documentclass[twoside,reqno]{amsart}

\usepackage{amssymb}
\usepackage{amsmath,amsfonts}
\usepackage[showonlyrefs]{mathtools}
\usepackage{bm}

\usepackage{tikz-cd}

\usepackage[a4paper,
	inner		= 2.5cm,
	outer		= 3.5cm,
	top			= 1.5cm,
	bottom		= 2.0cm,
	includeheadfoot
	]{geometry}

\usepackage[
	colorlinks	= true,
	allcolors 	= blue
	]{hyperref}

\usepackage{amsthm}
\theoremstyle	{plain}
\newtheorem		{theorem}					{Theorem}	[section]
\newtheorem		{lemma}			[theorem]	{Lemma}
\newtheorem		{corollary}		[theorem]	{Corollary}
\newtheorem		{proposition}	[theorem]	{Proposition}

\newtheorem		{assumption}	[theorem]	{Assumption}
\theoremstyle	{definition}
\newtheorem		{definition}	[theorem]	{Definition}
\newtheorem		{remark}		[theorem]	{Remark}
\newtheorem*	{remark*}					{Remark}

\numberwithin{equation}{section}

\newcommand{\nR}{{\mathbb R}}
\newcommand{\nC}{{\mathbb C}}
\newcommand{\nCR}{\nC\setminus\nR}

\newcommand{\cH}{{\mathcal H}}

\newcommand{\cL}{{\mathcal L}}

\newcommand{\cS}{{\mathcal S}}

   \def\sH{{\mathfrak H}}   
      
   \def\sN{{\mathfrak N}}

      \def\dC{{\mathbb C}}

   \def\dN{{\mathbb N}}   
      \def\dR{{\mathbb R}}

   \def\cB{{\mathcal B}}   
      
   \def\cH{{\mathcal H}}   
      \def\cL{{\mathcal L}}

\def\cS{{\mathcal S}}   \def\cT{{\mathcal T}}

\newcommand{\wt}{\widetilde}
\newcommand{\wh}{\widehat}

\DeclareMathOperator{\ran}{ran}

\DeclareMathOperator{\dom}{dom}
\DeclareMathOperator{\mul}{mul}

\DeclareMathOperator{\loc}{loc}
\DeclareMathOperator{\spn}{span}

\DeclareMathOperator{\supp}{supp}
\DeclareMathOperator{\comp}{comp}
\DeclareMathOperator{\clos}{clos}

\renewcommand{\Im}{\operatorname{Im}}

\newcommand*\conj[1]{\overline{#1}}

\newcommand{\dual}{\widehat S[R_1,R_2]}

\allowdisplaybreaks

\title{On a class of integral systems}

\begin{document}
\author{Volodymyr Derkach}
\address{
	Institut f\"ur  Mathematik,
	Technische Universit\"{a}t Ilmenau,
		Germany}
\address{
	Department of Mathematics,
	Vasyl Stus Donetsk National University,
	Vinnytsya, Ukraine}
\email{volodymyr.derkach@tu-ilmenau.de}

\author{Dmytro Strelnikov}
\address{
	Institut f\"ur  Mathematik,
	Technische Universit\"{a}t Ilmenau,
		Germany}
\email{dmytro.strelnikov@tu-ilmenau.de}

\author{Henrik Winkler}
\address{
	Institut f\"ur  Mathematik,
	Technische Universit\"{a}t Ilmenau,
		Germany}
\email{henrik.winkler@tu-ilmenau.de}

\thanks{The research of the first author was supported by
a grant of the German Research Foundation (DFG, grant TR 903/22-1) and
a grant of the Volkswagen Foundation.}

\begin{abstract}
We study spectral problems for two--dimensional integral system
with two  given non-decreasing functions $R_1$,  $R_2$ on an interval $[0,b)$
which is a generalization of the Krein string.
Associated to this system are the maximal linear relation $T_{\max}$ and the minimal linear relation $T_{\min}$
in the space $L^2(R_2)$
which are connected by $T_{\max}=T_{\min}^*$. It is shown that the limit point condition at $b$ for this system  is equivalent to the strong limit point condition
for the linear relation $T_{\max}$. In the limit circle case the strong limit point condition fails to hold on $T_{\max}$ but it is still satisfied on a subspace $T_N^*$ of $T_{\max}$ characterized by the Neumann boundary condition at $b$.
The notion of the principal Titchmarsh-Weyl coefficient of this integral system is introduced both in the limit point case and in the limit circle case.
Boundary triples for the linear relation $T_{\max}$ in the limit point case (and for $T_{N}^*$ in the limit circle case) are constructed and it is shown that the corresponding Weyl function coincides with the principal Titchmarsh-Weyl coefficient of the integral system.
The notion of the dual integral system  is introduced by reversing the order of  $R_1$ and  $R_2$.
It is shown that
the principal Titchmarsh-Weyl coefficients $q$ and $\wh q$ of the direct and the dual integral systems are related  by the equality
$\lambda \wh q(\lambda) = -1/q(\lambda)$
both in the regular and the singular case.

\end{abstract}

\subjclass{Primary 34B24; Secondary 34L05, 47A06, 47A57, 47B25, 47E05}

\date{\today}

\keywords{integral systems, Kre\u{\i}n strings, dual systems, principal Titchmarsh-Weyl coefficient, boundary triples, symmetric linear relations}

\maketitle

\section{Introduction} \label{sec:intro}

In this paper spectral problems for integral systems, associated dual systems and, in particular, Krein strings are investigated.
We consider an integral system of the form
\begin{equation}\label{eq:IntSys}
	u(x,\lambda) = u(0,\lambda) -J \int_0^x
	\begin{bmatrix} \lambda dR_2(t) & 0   \\  0 & dR_1(t)\end{bmatrix} u(t,\lambda),\quad J=\begin{bmatrix} 0 & -1   \\  1 & 0\end{bmatrix},\quad
\end{equation}
where $u = [u_1\ u_2]^T$, with some spectral parameter $\lambda\in\dC$ and measures $dR_1$ and $dR_2$ associated with non-decreasing functions $R_1(x)$ and $R_2(x)$ on an interval $[0,b)$, see~\cite{Ben89}.
If $R_1(x)\equiv x$ then $u_2=u_1'$ and system \eqref{eq:IntSys} is reduced to the equation of a vibrating string in the sense of M.\,G.~Krein
\begin{equation}\label{eq:Krein_S}
	u_1(x,\lambda) = u_1(0,\lambda) +x u_1'(0,\lambda) - \lambda \int_0^x (x-t) u_1(t,s)\, dR_2(t),
	\quad x\in[0,b).
\end{equation}
Integral systems \eqref{eq:IntSys} arise in the theory of diffusion processes with two measures
\cite{Man68,LaSc90}.
In the theory of stochastic processes the equation~\eqref{eq:Krein_S} describes generalized diffusion processes which includes both diffusion processes and birth and death processes ~\cite{Fel57,Fel59,IMcK65,Kas75}.
In mechanics the equation~\eqref{eq:Krein_S} describes small transverse oscillations of the string with the mass distribution function $R_2(x)$,~\cite{KacK68}.
Relations for the Titchmarsh-Weyl coefficients of Krein strings and their associated dual strings
were studied in~\cite{KacK68, KWW07}.

Let $c(\cdot,\lambda)$ and $s(\cdot,\lambda)$ be the unique solutions of~\eqref{eq:Krein_S} satisfying the initial conditions
\begin{equation}
	c(0,\lambda) = 1,\ c'(0,\lambda) = 0,
	\quad \text{and} \quad
	s(0,\lambda) = 0,\ s'(0,\lambda) = 1.
\end{equation}
The function
\begin{equation}\label{eq:m_NK}
	q_S(\lambda) \coloneqq \lim_{x\to b}\frac{s(x,\lambda)}{c(x,\lambda)}
\end{equation}
is called \emph{the principal Titchmarsh-Weyl coefficient of the string} \cite{KWW07} or \emph{the dynamic compliance coefficient} in the terminology of I.\,S.~Kac and M.\,G.~Krein \cite{KacK68} and describes the spectral properties of the string. The principal Titchmarsh-Weyl coefficient $q(\lambda)$ is a Stieltjes function and the measure $d\sigma$ from its integral representation
\begin{equation}\label{eq:Int_Rep_mS}
	q_S(\lambda)=a+\int_0^\infty\frac{d\sigma(t)}{t-\lambda}, \quad a\ge 0
\end{equation}
is the spectral measure of the string $S_1[b,R_2]$, which in the limit point case is given by the boundary condition $u'(0)=0$.

Denote the integral system~\eqref{eq:IntSys} by $S[R_1,R_2]$.
In the present paper we define \emph{the principal Titchmarsh-Weyl coefficient} $q$ of the integral system $S[R_1,R_2]$ by
\begin{equation}\label{eq:m_TW}
	q(\lambda) \coloneqq \lim_{x\to b}\frac{s_1(x,\lambda)}{c_1(x,\lambda)},
\end{equation}
where
$c_1(\cdot, \lambda)$, $c_2(\cdot, \lambda)$ and $s_1(\cdot, \lambda)$, $s_2(\cdot, \lambda)$
are pairs of the unique solutions of~\eqref{eq:IntSys} satisfying the initial conditions
\begin{equation} \label{eq:csK}
	c_1(0,\lambda) = 1,\ c_2(0,\lambda) = 0,
	\quad \text{and} \quad
	s_1(0,\lambda) = 0,\ s_2(0,\lambda) = 1.
\end{equation}
Formula~\eqref{eq:m_TW} requires justification. For this purpose we use the operator approach
to the integral system $S[R_1,R_2]$ developed in~\cite{Str19}, the boundary triples technique from~\cite{Koc75,GG91} and the theory of associated Weyl functions as introduced in ~\cite{DM91,DM95}.
The maximal linear relation $T_{\max}$ is defined as the set of pairs
$\bm{u} = [u_1\ f]^T$ such that $u_1,f\in L^2(R_2)$ and the equation~\eqref{eq:sys_lin_rel} is satisfied for some $u_2\in BV_{\loc}[0,b)$, see Definition~\ref{def:Tmax}.
The closure of its restriction to the set of compactly supported functions is called the minimal linear relation $T_{\min}$.
 In~\cite{Str19} it is shown that $T_{\min}$ is symmetric in $L^2(R_2)$, $T_{\max}=T_{\min}^*$
 and boundary triples for the linear relation $T_{\min}$ were constructed both in the limit point and in the limit circle case.

In Theorem~\ref{thm:SLP} we show that the system $S[R_1,R_2]$
is in the limit point case at $b$ if and only if it satisfies \emph{the strong limit point condition} at $b$, see~\cite{Ever76}, which in our case is of the form
\begin{equation}\label{eq:SLP}
	\lim_{x\to b}u_1(x)u_2(x) = 0
	\quad \text{for all} \quad \bm{u} \in T_{\max}.
\end{equation}
As a consequence of~\eqref{eq:SLP} we conclude that in the limit point case the linear relation $T_{\min}$ and its von Neumann extension $A_N$, characterized by the boundary condition $u_2(0)=0$, are nonnegative,
the corresponding Weyl function is a Stieltjes function and coincides with the principal Titchmarsh-Weyl coefficient of the system $S[R_1,R_2]$. The strong limit point condition for second order differential operators was introduced by W. Everitt~\cite{Ever76}.

In the limit circle case the linear relation $T_{\min}$ has defect numbers $(2,2)$, in this case an intermediate symmetric extension $T_N$ with defect numbers $(1,1)$ of $T_{\min}$ is considered as the restriction of $T_{\max}$ to the set of elements $\bm{u}\in T_{\max}$ such that $u_1(0)=u_2(0)=u_2(b)=0$.
In this case we show in Lemma~\ref{lem:psi} that the strong limit point condition
\eqref{eq:SLP} fails to hold, but still the limit in~\eqref{eq:SLP} is vanishing on the subspace $T_{N}^*$ of $T_{\max}$, i.e.
the following \emph{Evans--Everitt condition} holds, cf.~\cite{EvEver}:
\begin{equation}\label{eq:Ever}
	\lim_{x\to b} u_1(x)u_2(x) = 0
	\quad \text{for all} \quad \bm{u}\in T_{N}^*.
\end{equation}
This result implies the nonnegativity of the linear relation $T_{N}$ and its
selfadjoint extension
\begin{equation*}
	A_N = \{\bm{u} \in T_{\max} \colon u_2(0) = u_2(b) = 0\}.
\end{equation*}
In~\cite{Kost13} another analytical object ---
\emph{the Neumann $m$-function} of the system $S[R_1,R_2]$ was introduced by the equality
\begin{equation}\label{eq:m_N}
	m_N(\lambda) \coloneqq \lim_{x\to b} \frac{s_2(x,\lambda)}{c_2(x,\lambda)},
\end{equation}
which is a special case of a more general definition of the Neumann $m$-function presented in~\cite{Ben89}.
In Proposition~\ref{prop:BT_IS_N} it is shown that the Neumann $m$-function $m_N(\lambda)$ is a Stieltjes function and it coincides with the principal Titchmarsh-Weyl coefficient of the integral system $S[R_1,R_2]$.

The system $S[R_1,R_2]$
is called \emph{regular} if $R_1(b)+R_2(b)<\infty$ and \emph{singular} otherwise.
In the regular case we construct \emph{the canonical singular extension} $S[\wt R_1,\wt R_2]$ of the system $S[R_1,R_2]$ with $R_1, R_2$ extended to non-decreasing functions $\wt R_1, \wt R_2$ on the interval $(0,\infty)$, so that the principal Titchmarsh-Weyl coefficients of both systems coincide.

\emph{The dual system} $\dual$ of the integral system $S[R_1,R_2]$ in the singular case is obtained by changing the roles of $R_1$ and $R_2$.
In the regular case the dual system of the integral system $S[R_1,R_2]$ is defined as the dual of the canonical singular extension $S[\wt R_1,\wt R_2]$ of the system $S[R_1,R_2]$.
The main result of the paper is Theorem~\ref{thm:2} where it is shown that
the principal Titchmarsh-Weyl coefficient $\wh q$ of the dual system is related to the principal Titchmarsh-Weyl coefficient $q$ of the system $S[R_1,R_2]$
by the equality
\begin{equation}\label{eq:mS_whmS}
	\wh q(\lambda) = -\frac{1}{\lambda q(\lambda)}.
\end{equation}
both in the regular and the singular case.

In the case of a string ($R_1(x)=x$) the notion of the dual string and the formula~\eqref{eq:mS_whmS}
connecting the principal Titchmarsh-Weyl coefficients of the direct and the dual string in the singular case was presented in~\cite{KacKr58}. In~\cite{KWW07} some further relations between strings, dual strings and canonical systems of differential equations were studied. Analogues of these relations between integral systems and canonical systems can also be established and will be presented in a forthcoming paper.

\section{Preliminaries}
\subsection{Linear relations}
\label{subsec:pre:lr}
Let $\sH$ be a Hilbert space.
 A linear relation $T$ in $\sH$
is a linear subspace of $\sH \times \sH$.
Let us recall some basic definitions and
properties associated with linear relations
in~\cite{Arens,Ben72}.

The \emph{domain}, the \emph{range}, the \emph{kernel}, and the \emph{multivalued part} of a linear relation $T$ are defined as follows:
\begin{align}
	\dom{T} &\coloneqq \left\{ f \colon \begin{bmatrix} f \\ g \end{bmatrix} \in T \right\}, &
	\ran{T} &\coloneqq \left\{ g \colon \begin{bmatrix} f \\ g \end{bmatrix} \in T \right\},\\
	\ker{T} &\coloneqq \left\{ f \colon \begin{bmatrix} f \\ 0 \end{bmatrix} \in T \right\}, &
	\mul{T} &\coloneqq \left\{ g \colon \begin{bmatrix} 0 \\ g \end{bmatrix} \in T \right\}.
\end{align}
The \emph{adjoint} linear relation $T^*$ is defined by
\begin{equation}
	T^* \coloneqq \left\{
	\begin{bmatrix} u \\ f \end{bmatrix}
	\in \sH \times \sH \colon \langle f,v\rangle_{\sH} = \langle u,g\rangle_{\sH}\ \text{for any}\
	\begin{bmatrix} v \\ g \end{bmatrix}
	\in T \right\}.
\end{equation}
A linear relation $T$ in $\sH$ is called \emph{closed} if $T$ is closed as a subspace of $\sH \times \sH$.
The set of all closed linear operators (relations) is denoted by $\mathcal{C}(\sH)$ ($\wt{\mathcal{C}}(\sH)$).
Identifying a linear operator $T \in \mathcal{C}(\sH)$ with its graph one can consider $\mathcal{C}(\sH)$ as a part of $\wt{\mathcal{C}}(\sH)$.

Let $T$ be a closed linear relation, $\lambda \in \mathbb{C}$, then
	\begin{equation}
		T - \lambda I \coloneqq \left\{ \begin{bmatrix} f \\ g-\lambda f \end{bmatrix} \colon
		\begin{bmatrix} f \\ g \end{bmatrix}
		\in T \right\}.
	\end{equation}
	A point $\lambda \in \mathbb{C}$ such that $\ker{\left( T - \lambda I \right)} = \{0\}$ and $\ran{\left( T - \lambda I \right)} = \sH$ is called a \emph{regular point} of the linear relation $T$. Let $\rho(T)$ be the set of regular points.
	The \emph{point spectrum} $\sigma_p(T)$ of the linear relation $T$
	is defined by
	\begin{equation}\label{eq:Point_s}
		\sigma_p(T) \coloneqq
		\{\lambda\in\mathbb{C} \colon \ker(T-\lambda I)\ne\{0\}\},
	\end{equation}

A linear relation $T$ is called \emph{symmetric} if $T \subseteq T^*$.
A point $\lambda \in \mathbb{C}$ is called a \emph{point of regular type} (and is written as $\lambda \in \widehat{\rho}(T)$) for a closed symmetric linear relation $T$, if $\lambda \notin \sigma_p(T)$ and the subspace $\ran(T-\lambda I)$ is closed in $H$.
For $\lambda \in \widehat{\rho}(T)$ let us set
$\sN_{\lambda}(T^*) \coloneqq \ker (T^* -\lambda I)$ and
\begin{equation}
	\widehat{\sN}_{\lambda}(T^*)  \coloneqq
	\left\{
	\bm{u}_{\lambda} =
	\begin{bmatrix} u_{\lambda} \\ \lambda u_{\lambda} \end{bmatrix} \colon u_{\lambda} \in \sN_{\lambda}(T^*)
	\right\}.
\end{equation}
The \emph{deficiency indices} of a symmetric linear relation $T$ are defined as
\begin{equation}
	n_{\pm}(T) \coloneqq \dim \ker (T^* \mp iI).
\end{equation}

\subsection{Boundary triples and Weyl functions}
\label{subsec:pre:triples}
Let $T$ be a symmetric linear relation with deficiency indices $(1,1)$.
In the case of a densely defined operator the notion of the boundary triple was introduced in \cite{Koc75,GG91}.
Following the papers \cite{M92,DM95} we shall give a definition of a boundary triple for the linear relation $T^*$.

\begin{definition} \label{def:btriple}
	A tuple $\Pi = (\dC,\Gamma_0,\Gamma_1)$, where $\Gamma_0$ and $\Gamma_1$ are linear mappings from $T^*$ to $\dC$, is called a \emph{boundary triple} for the linear relation $T^*$, if:
	\begin{enumerate}
	\item [(i)]
	for all
	$\bm{u} = \begin{bmatrix} u \\ f \end{bmatrix}$,
	$\bm{v} = \begin{bmatrix} v \\ g \end{bmatrix} \in T^*$
	the following generalized Green's identity holds
	\begin{equation} \label{eq:1.9}
		\langle f,v \rangle_{\sH} - \langle u,g \rangle_{\sH} =\Gamma_1 \bm{u} \conj{\Gamma_0\bm{v}} -\Gamma_0 \bm{u} \conj{\Gamma_1\bm{v}};
	\end{equation}
	\item [(ii)]
	the mapping
	$\Gamma=\begin{bmatrix}\Gamma_0 \\ \Gamma_1\end{bmatrix} \colon	T^* \rightarrow \dC^2$
	is surjective.
	\end{enumerate}
\end{definition}

Notice, that in contrast to~\cite{M92} the linear relation $T$ is not supposed to be single-valued.
The following linear relations
\begin{equation} \label{e q:A0A1}
	A_0 \coloneqq \ker \Gamma_0, \qquad A_1 \coloneqq \ker \Gamma_1
\end{equation}
are selfadjoint extensions of the symmetric linear relation $T$.

\begin{definition}[\cite{DM91,DM95}] \label{def:M_gamma}
	Let $\Pi = (\dC,\Gamma_0,\Gamma_1)$ be a boundary triple for the linear relation $T^*$.
	The scalar function $m(\cdot)$ and the vector valued function $\gamma(\cdot)$ defined by
	\begin{equation}\label{def:Weyl}
	 	m(\lambda) \Gamma_0 \bm{u}_\lambda = \Gamma_1 \bm{u}_\lambda, \quad
	 	\gamma(\lambda) \Gamma_0 \bm{u}_\lambda = u_{\lambda}, \quad
	 	\bm{u}_\lambda=
	\begin{bmatrix} u_{\lambda} \\ \lambda u_{\lambda} \end{bmatrix} \in \widehat{\sN}_{\lambda}(T^*), \quad
	 	\lambda \in \rho(A_0)
	\end{equation}
	are called \emph{the Weyl function} and \emph{the $\gamma$-field} of the symmetric linear relation $T$ corresponding to the boundary triple $\Pi$.
\end{definition}

The Weyl function and the $\gamma$-field are connected via the next identity (see \cite{DM95})
\begin{equation} \label{eq:weyl-id}
	m(\lambda) - m(\zeta)^* =
	(\lambda - \conj{\zeta}) \gamma(\zeta)^* \gamma(\lambda),
	\quad
	\lambda,\zeta \in \rho(A_0).
\end{equation}

\begin{definition}[\cite{KaKr74}]
	A function
	$m \colon \mathbb{C}\setminus\dR\to \cB (\cH)$ is said
	to be a \emph{Herglotz-Nevanlinna function}, if the following conditions hold:
	\begin{enumerate}
		\item [(i)]
		$m$ is holomorphic in $\mathbb{C}\setminus\dR$;
		\item [(ii)]
		$\Im m(\lambda) \geq 0$ as $\lambda \in \mathbb{C}_+:=\{\lambda\in\dC: \text{Im }\lambda>0\}$;
		\item [(iii)]
		$m(\conj{\lambda}) = m^*(\lambda)$ for $\lambda \in \mathbb{C}\setminus\dR$.
	\end{enumerate}
\end{definition}

It follows from~\eqref{eq:weyl-id} that the Weyl function $m(\cdot)$ is a Herglotz-Nevanlinna function.
A Herglotz-Nevanlinna function $m$ which admits a holomorphic continuation to $\dR_-$
and takes nonnegative values for all $\lambda\in\dR_-$ is called a \emph{Stieltjes function}.
Every Stieltjes function $m$ admits an integral representation
\eqref{eq:Int_Rep_mS} with a non-decreasing function $\sigma(t)$ such that $\int_{\dR_+}(1+t)^{-1}d\sigma(t)<\infty$.

\subsection{Minimal and maximal relations associated with the integral system $S[R_1,R_2]$ }

Let $I=[0,b)$ be an interval with $b\le\infty$, let $R(x)$ be a non-decreasing left-continuous function on $I$ such that $R(0)=0$, let $dR$ be the corresponding Lebesgue-Stieltjes measure, and let $\cL^2(R,I)$ be an inner product space which consists of complex valued functions $f$ such that
\begin{equation}
	\int_I |f(t)|^2\, dR(t) < \infty
\end{equation}
with inner product defined by
\begin{equation}
	\langle f,g \rangle_{R} = \int_I f(t) \conj{g(t)} dR(t).
\end{equation}
$\cL^2_{\comp}(R,I)$ denotes the subspace consisting of those $f\in \cL^2(R,I)$ with compact support in $I$, $BV[0,b)$ denotes the set of functions of bounded variation on $[0,b)$ and $BV_{\loc}[0,b)$ is the set of functions $f$ such that $f\in BV[0,b')$ for every $b'<b$.
Denote by $L^2(R,I)$ the corresponding quotient space for $\cL^2(R,I)$, which consists of equivalence classes w.r.t. $dR$ and denote by $\pi$ the corresponding quotient map, i.e. $\pi \colon \cL^2(R,I) \to L^2(R,I)$. Further we omit $I$ in the notation if it coincides with $[0,b)$.

From now on the following convention is used for the integration limits for any measure $d\sigma$ on an interval:
\begin{equation}
	\int_a^x f\, d\sigma \coloneqq \int_{[a,x)} f\, d\sigma.
\end{equation}
Thus, an integral as a function of its upper limit is always left-continuous.
With every function of bounded variation $f$ we associate the left-continuous and the right-continuous functions $f_-$ and $f_+$ defined by
\begin{equation}
	f_-(x) \coloneqq \lim_{t\uparrow x} f(t), \quad
	f_+(x) \coloneqq \lim_{t\downarrow x} f(t).
\end{equation}

Let $u$ and $v$ be left-continuous functions of bounded variation, $du$ and $dv$ be the corresponding Lebesgue-Stieltjes measures.
The following integration-by-parts formula for the Lebesgue-Stieltjes integral (see e.g. \cite{Hew60}) is used throughout the paper
\begin{equation} \label{eq:in-by-parts}
	\int_a^x u\, dv + \int_a^x v_+\, du = u(x)v(x) - u(a)v(a).
\end{equation}

If $u$ and $u_+$ have no zeros then it follows with $v=1/u$ from~\eqref{eq:in-by-parts}
\begin{equation*}
	d(1) = d \left( \frac{u}{u} \right) =
	u\, d\left( \frac{1}{u} \right) + \frac{1}{u_+}\, du = 0.
\end{equation*}
This leads to the quotient-rule formula
\begin{equation} \label{eq:quotient-rule}
	d\left( \frac{1}{u} \right) = -\frac{du}{u u_+}.
\end{equation}

The following existence and uniqueness theorem for integral systems was proven in \cite[Theorem 1.1]{Ben89}.

\begin{theorem}\label{thm:Ex_Uniq}
Let $dS$ be a complex $n\times n$ matrix-valued measure. For every left continuous (either
$n\times n$ or $n\times1$ matrix valued) function $A(x)$ in $BV_{\loc}[0,b)$ there is a unique function $U$ such that the equality
	\begin{equation}\label{eq:sys_SolvTH}
		U(x) = A(x) + \int_0^x dS \cdot U
	\end{equation}
	holds for every point $x\in[0,b)$.
\end{theorem}
\begin{remark}\label{rem:2.6}
	Due to the properties of the Lebesgue-Stieltjes integral and the used convention, any solution $U$ to~\eqref{eq:sys_SolvTH} is left continuous and belongs to $BV_{\loc}[0,b)$, componentwise.
\end{remark}

Now we focus on integral systems $S[R_1,R_2]$ of the form~\eqref{eq:IntSys},
where $R_1(x)$ and $R_2(x)$ are nondecreasing and left-continuous real-valued functions on the interval $I=[0,b)$ such that $R_1(0)=R_2(0)=0$. We define the corresponding inhomogeneous system.

\begin{definition} \label{def:A_triple}
	Let $f \in \cL^2(R_2)$ and $[u_1\ u_2]^T$ be a vector-valued function such that the following equation
	\begin{equation} \label{eq:sys_lin_rel}
		\begin{bmatrix} u_1\\u_2 \end{bmatrix} (x)
		=
		\begin{bmatrix} u_1\\u_2 \end{bmatrix} (0)
		- J \int_0^x
		\begin{bmatrix}
			dR_2 & 0 \\
			0    & dR_1
		\end{bmatrix}
		\begin{bmatrix} f\\u_2 \end{bmatrix}
	\end{equation}
	holds for every point $x\in[0,b)$.
	The triple $(u_1,u_2,f)$ is said to belong to the set $\cT$
	if $u_1 \in \cL^2(R_2)$.
\end{definition}
Due to Remark~\ref{rem:2.6} for every $(u_1,u_2,f)\in\cT$ both functions $u_1$ and $u_2$ belong to $BV_{\loc}[0,b)$. Theorem~\ref{thm:Ex_Uniq} implies that for every $f \in \cL^2(R_2)$ the vector-valued function $[u_1\ u_2]^T$ satisfying~\eqref{eq:sys_lin_rel} is uniquely determined by its initial values at zero, however $u_1 \in \cL^2(R_2)$ is not guaranteed for an arbitrary $f$.

\begin{definition}\label{def:Tmax}
	We define the maximal and the pre-minimal relations $T_{\max}$, $T'\subset L^2(R_2) \times L^2(R_2)$ by
	\begin{equation} \label{eq:Pi_map}
		T_{\max} \coloneqq \left\{
		\bm{u} = \begin{bmatrix} \pi u_1 \\ \pi f \end{bmatrix} \colon
		(u_1, u_2, f) \in \cT \right\}.
	\end{equation}
	\begin{equation} \label{eq:Pi_min}
		T' \coloneqq \left\{
		\bm{u} = \begin{bmatrix} \pi u_1 \\ \pi f \end{bmatrix}
		\in T_{\max} \colon
		(u_1, u_2, f) \in \cT, \ u_1, f\in L^2_{\comp}(R_2,I)
		\right\}.
	\end{equation}
\end{definition}

Everywhere in the paper, except Remark~\ref{rem:MP}, we suppose that the following two natural assumptions hold.
\begin{assumption} \label{assum:R1R2}
	The functions $R_1$ and $R_2$ have no common points of discontinuity.
\end{assumption}

\begin{assumption} \label{assum:surj}
	There exists an interval $[0,b_0) \subseteq [0,b)$ such that
	\begin{equation} \label{eq:assum_cond}
		\dim\spn\{\pi 1,\pi R_1\} = 2
	\end{equation}
	where $\pi \colon \cL^2(R_2,[0,b_0)) \to L^2(R_2,[0,b_0))$ is the corresponding quotient map.
\end{assumption}

Assumption~\ref{assum:R1R2} has the important consequence that the first component of a solution has no discontinuity in common with the second component of any solution $(u_1, u_2, f) \in \cT$.
Assumption~\ref{assum:surj} makes it possible to assign correctly the values $u_1(x)$ and $u_2(x)$ for every $\bm{u} \in T_{\max}$. In case of absolutely continuous functions $R_1$ and $R_2$ the equivalent to $S[R_1,R_2]$ differential system is \emph{definite} in the sense of \cite[Definition 2.14]{LM03} if and only if Assumption~\ref{assum:surj} holds.
\begin{definition}
	Let $(u_1,u_2,f) \in \cT$ and $
\bm{u} \in T_{\max}$ be its image under the mapping
\begin{equation}\label{eq:Pi_map1}
	\cT \ni (u_1,u_2,f) \mapsto
	\bm{u} = \begin{bmatrix} \pi u_1 \\ \pi f \end{bmatrix}
	\in T_{\max}.
\end{equation}
 The mappings $\phi_{1,2}[x] \colon T_{\max} \to \nC$ are defined by
	\begin{equation}
		\phi_i [x] \bm{u} \coloneqq  u_i (x), \quad i \in \{1,2\},\quad x\in[0,b).
	\end{equation}
\end{definition}
The following Proposition provides an analog of \cite[Proposition 2.15]{LM03} for integral system $S[R_1,R_2]$.
\begin{proposition}
If Assumptions~\ref{assum:R1R2} and~\ref{assum:surj} hold then
	the mappings $\phi_{1,2}[x]$ are well-defined.
\end{proposition}
\begin{proof}
	In general, the mapping defined by~\eqref{eq:Pi_map1} is not invertible.
	Suppose that $(u_1,u_2,f)$ and $(\wt{u}_1,\wt{u}_2,\wt{f})$ are two pre-images of
	$\bm{u} = \begin{bmatrix} \pi u_1 \\ \pi f \end{bmatrix} \in T_{\max}$
	as it is shown on the following diagram.
	\begin{center}
	\begin{tikzcd}
		(u_1,u_2,f) \arrow{d} \arrow{r} &
		\bm{u} \arrow[dashed]{d} &
		\arrow{l} (\wt{u}_1,\wt{u}_2,\wt{f}) \arrow{d} \\
		u_{1,2}(x) \arrow{r} &
		\phi_{1,2}[x] \bm{u} &
		\arrow{l} \wt{u}_{1,2}(x)
	\end{tikzcd}
	\end{center}
	Let us show that due to Assumption~\ref{assum:surj}
	\begin{equation}
		u_1(x) = \wt{u}_1(x), \quad u_2(x) = \wt{u}_2(x), \quad \text{for}\ t \in [0,b).
	\end{equation}
	Clearly, $(u_1-\wt{u}_1, u_2-\wt{u}_2, Y-\wt{Y}) \in \cT$. Taking into account $\pi u_1 = \pi \wt{u}_1$, $\pi f = \pi \wt{f}$, it follows from~\eqref{eq:sys_lin_rel} that
	\begin{equation} \label{eq:y1y2_recovery}
		\begin{bmatrix}
			u_1(x) - \wt{u}_1(x)\\ u_2(x) - \wt{u}_2(x)
		\end{bmatrix}
		=
		\begin{bmatrix}
			(u_1(0) - \wt{u}_1(0)) + (u_2(0) - \wt{u}_2(0))R_1(x)\\
			u_2(0) - \wt{u}_2(0)
		\end{bmatrix}.
	\end{equation}
	The mapping $\pi$ to applied the first line of~\eqref{eq:y1y2_recovery} gives
	\begin{equation}
		0 = (u_1(0) - \wt{u}_1(0)) \cdot \pi 1 + (u_2(0) - \wt{u}_2(0)) \cdot \pi R_1.
	\end{equation}
	Now it follows from~\eqref{eq:assum_cond} that $u_1(0) = \wt{u}_1(0)$, $u_2(0) = \wt{u}_2(0)$, which together with~\eqref{eq:y1y2_recovery} completes the proof.
\end{proof}

Further in the text we will simply write $u_{1,2}(x)$ instead of $\phi_{1,2}[x]\bm{u}$ unless this can lead to confusion.
For a pair of vector-valued functions
$u = \begin{bmatrix} u_1 & u_2 \end{bmatrix}^T$,
$v = \begin{bmatrix} v_1 & v_2 \end{bmatrix}^T$
we define the generalized Wronskian by
\begin{equation}
	[u,v](x) \coloneqq u_1(x) v_2(x) - u_2(x) v_1(x).
\end{equation}

\begin{proposition} \label{prop:green_formulas}
If $(u_1,u_2,f)$ and $(v_1,v_2,g)$ belong to $\cT$
then the following generalized first and second Green's identities hold
	\begin{equation}\label{eq:G1}
		\int_0^x f v_1\, dR_2 = \int_0^x u_2 v_2\, dR_1 - u_2(x) v_1(x) + u_2(0) v_1(0),
	\end{equation}
	\begin{equation}\label{eq:G2}
		\int_0^x (f v_1 - u_1 g) \, dR_2 = [u,v](x) - [u,v](0).
	\end{equation}
	for an arbitrary interval $[0,x) \subset [0,b)$.
\end{proposition}
\begin{proof}
	We recall that due to Assumption~\ref{assum:R1R2} the functions $R_1$ and $R_2$ do not have common points of discontinuity, so neither do the functions $v_1$ and $u_2$.
	By virtue of~\eqref{eq:sys_lin_rel} we get
	\begin{equation}
			dv_1 = v_2\, dR_1, \quad du_2 = - f\, dR_2.
	\end{equation}
	and hence, using the integration-by-parts formula~\eqref{eq:in-by-parts},
	\begin{equation} \label{eq:2.7}
		d(u_2 v_1) =
		v_1\, du_2  + u_2\, dv_1 = u_2 v_2\, dR_1 - f v_1\, dR_2.
	\end{equation}
	Integrating \eqref{eq:2.7} over $[0,x)$ provides~\eqref{eq:G1}.
	Swapping the tuples $(u_1,u_2,f)$ and $(v_1,v_2,g)$ in~\eqref{eq:2.7} and subtracting the obtained expression from~\eqref{eq:2.7} proves~\eqref{eq:G2}.
\end{proof}

Theorem~\ref{thm:Ex_Uniq} provides that system $S[R_1,R_2]$ has a unique solution for every choice of initial values.
Let $c(\cdot,\lambda) = [c_1(\cdot,\lambda)\ c_2(\cdot,\lambda)]^T$ and $s(\cdot,\lambda) = [s_1(\cdot,\lambda)\ s_2(\cdot,\lambda)]^T$
be its unique solutions satisfying the initial conditions~\eqref{eq:csK}.
\begin{corollary}
	For every $\lambda\in\dC$ and $x\in [0,b)$ the following formulas hold:
	\begin{equation} \label{eq:Liouv}
		[c(\cdot,\lambda),s(\cdot,\lambda)](x) = c_1(x,\lambda) s_2(x,\lambda) - c_2(x,\lambda) s_1(x,\lambda) = 1,
	\end{equation}
	\begin{equation} \label{eq:Liouv+}
		c_{1+}(x,\lambda) s_2(x,\lambda) - c_2(x,\lambda) s_{1+}(x,\lambda) = 1,
	\end{equation}	
	\begin{equation} \label{eq:Liouv++}
		c_1(x,\lambda) s_{2+}(x,\lambda) - c_{2+}(x,\lambda) s_1(x,\lambda) = 1.
	\end{equation}
	\end{corollary}
\begin{proof}
	Equality~\eqref{eq:Liouv} follows immediately from either~\eqref{eq:Kernel_U} or \eqref{eq:G2}. Further we subtract the left-hand side of \eqref{eq:Liouv} from the left-hand side of \eqref{eq:Liouv+}:
	\begin{multline} \label{eq:cor_Liouv}
		(c_{1+}(x,\lambda) s_2(x,\lambda) - c_2(x,\lambda) s_{1+}(x,\lambda)) -
		(c_1(x,\lambda) s_2(x,\lambda) - c_2(x,\lambda) s_1(x,\lambda)) =\\
		(c_{1+}(x,\lambda)-c_1(x,\lambda)) s_2(x,\lambda) -
		c_2(x,\lambda) (s_{1+}(x,\lambda)- s_1(x,\lambda))
	\end{multline}
	One can immediately see that the expression~\eqref{eq:cor_Liouv} is equal to zero at every point of continuity of $R_1$. Let $x_0$ be a point of discontinuity of $R_1$. From~\eqref{eq:sys_lin_rel} one can see that
	\begin{align}
		c_{1+}(x_0,\lambda)-c_1(x_0,\lambda) = c_2(x_0, \lambda)\, dR_1(\{x_0\}),\\
		s_{1+}(x_0,\lambda)-s_1(x_0,\lambda) = s_2(x_0, \lambda)\, dR_1(\{x_0\})
	\end{align}
	and hence
	\begin{multline}
		(c_{1+}(x_0,\lambda)-c_1(x_0,\lambda)) s_2(x_0,\lambda) -
		c_2(x_0,\lambda) (s_{1+}(x_0,\lambda)- s_1(x_0,\lambda)) = \\
		c_2(x_0,\lambda) s_2(x_0,\lambda)\, dR_1(\{x_0\}) -
		s_2(x_0,\lambda) c_2(x_0,\lambda)\, dR_1(\{x_0\}) = 0.
	\end{multline}
	The proof of~\eqref{eq:Liouv++} is similar.
\end{proof}

It follows from~\eqref{eq:G2} that the pre-minimal relation $T'$ is symmetric in $L^2(R_2)$.
\begin{definition} \label{def:Amin}
The minimal relation $T_{\min}$ is defined as the closure of the pre-minimal linear relation $T'$: $T_{\min} = \clos{T'}$.
\end{definition}
As was shown in~\cite{Str19} the linear relation $T_{\min}$ is also symmetric, $T_{\min}^* = T_{\max}$ and
	\begin{equation} \label{eq:def-Amin}
		T_{\min} \coloneqq
		\left\{ \bm{u} = \begin{bmatrix} \pi u_1 \\ \pi f \end{bmatrix} \in T_{\max} \colon
		u_1(0) = u_2(0) = [u,v]_b = 0\ \text{for all}\
		\bm{v} = \begin{bmatrix} \pi v_1 \\ \pi g \end{bmatrix} \in T_{\max} \right\}.
	\end{equation}

\begin{lemma} \label{lem:Weyl_disks}
Let $l<b$, $h\in\clos{\nC_+}\cup\{\infty\}$, and let $m(\lambda,l,h)$ be some coefficient such that the function
\begin{equation} \label{eTWeyl}
	\psi(t,\lambda) \coloneqq s(t,\lambda)-m(\lambda,l,h)\, c(t,\lambda)
\end{equation}
satisfies the condition $\psi_1(l,\lambda) + h \psi_2(l,\lambda) = 0$.
Then:
\begin{enumerate}
	\item [(i)] The coefficient $m$ is well-defined and can be calculated as
\begin{equation} \label{eq:mlhl}
	m(\lambda,l,h) =
	\frac{s_1(l,\lambda) + h s_2(l,\lambda)}{c_1(l,\lambda) + h c_2(l,\lambda)}.
\end{equation}
	\item [(ii)] For every $\lambda \in \nC_+$ the set
	$D_l(\lambda) \coloneqq \{m(\lambda,l,h) \colon h\in\clos{\dC_+}\cup\{\infty\}\}$ is a disk in $\nC_+$ such that $\omega \in D_l(\lambda)$ if and only if
\begin{equation} \label{eq:WeylD}
	\int_0^l |s_1(t,\lambda) - \omega c_1(t,\lambda)|^2 dR_2(t) \le
	\frac{\Im\omega}{\Im\lambda},
\end{equation}
and its radius can be calculated as
\begin{equation}\label{Abh.13.10}
	r_l(\lambda) = \left(2 \Im \lambda\int_0^l |s_1(t,\lambda)|^2 dR_2(t)\right)^{-1}.
\end{equation}

    \item [(iii)] The Weyl discs $D_l(\lambda)$ are nested, i.e. $D_{l_2} \subseteq D_{l_1}$ provided $l_1 < l_2 < b$, and the function $s_1(\cdot,\lambda) - \omega c_1(\cdot,\lambda)$ belongs to $\cL^2(R_2)$ provided $\omega\in\cap_{l<b} D_l(\lambda)$.
\end{enumerate}
\end{lemma}
\begin{proof}
	(i) From~\eqref{eTWeyl} and the condition $\psi_1(l,\lambda) + h \psi_2(l,\lambda) = 0$ we get
	\begin{equation}
		\psi_1(l,\lambda)  + h \psi_2(l,\lambda) =
		(s_1(l,\lambda) + h s_2(l,\lambda)) - m(\lambda,l,h) (c_1(l,\lambda) + h c_2(l,\lambda)) = 0
	\end{equation}
	which results as~\eqref{eq:mlhl}.

	(ii) It is clear from formula~\eqref{eq:mlhl} that the function $m(\lambda,l,\cdot)$ maps $\dR_+\cup\{\infty\}$ into a circle.
	Let $h\in\clos{\dC_+}\cup\{\infty\}$ and $\omega \in D_l(\lambda)$.
	Applying the second Green's identity~\eqref{eq:G2} to the tuples $\{\psi_1,\psi_2,\lambda\psi_1\}$ and $\{\psi_1^*,\psi_2^*,\conj{\lambda}\psi_1^*\}$ provides
	\begin{equation}
		(\lambda - \conj{\lambda}) \int_0^l |\psi_1(t,\lambda)|^2 dR_2(t) =
		(\omega - \conj{\omega}) - (h - \conj{h}) |\psi_2(l)|^2
	\end{equation}
	and hence
	\begin{equation} \label{eq:wde}
		\int_0^l |s_1(t,\lambda) - \omega c_1(t,\lambda)|^2\, dR_2(t)
		=
		\frac{\Im\omega}{\Im\lambda} - \frac{\Im h}{\Im\lambda} |\psi_2(l)|^2.
	\end{equation}
	Since $\Im h > 0$, \eqref{eq:WeylD} follows now from \eqref{eq:wde}.

	(iii) Let $l_1 < l_2 < d$ and let $\omega \in D_{l_2}$. Then
	\begin{equation}
		\int_0^{l_1} |s_1(t,\lambda) - \omega c_1(t,\lambda)|^2\, dR_2(t) \le
		\int_0^{l_2} |s_1(t,\lambda) - \omega c_1(t,\lambda)|^2\, dR_2(t) \le
		\frac{\Im\omega}{\Im\lambda}
	\end{equation}
	and therefore $D_{l_2} \subseteq D_{l_1}$. Assume now $\omega\in\cap_{l<b}D_l(\lambda)$. Passing to the limit as $l\to b$ in~\eqref{eq:WeylD}, one gets
	\begin{equation}
		\int_0^b |s_1(t,\lambda) - \omega c_1(t,\lambda)|^2\, dR_2(t) \le
		\frac{\Im\omega}{\Im\lambda},
	\end{equation}
	which proves that $s_1(t,\lambda) - \omega c_1(t,\lambda) \in \cL^2(R_2)$.
\end{proof}

Assume that the point $b$ is singular for the system~\eqref{eq:IntSys}. Then the following alternative holds:
\begin{enumerate}
	\item [(i)]
	either the discs $D_l(\lambda)$ shrink to a limit point as $l\to b$ and then $\dim\sN_{\lambda} (T_{\max}) = 1$;
	\item [(ii)]
	or the discs $D_l(\lambda)$ converge to a limit disc as $l\to b$ and then  $\dim\sN_{\lambda} (T_{\max}) = 2$.
\end{enumerate}

\begin{definition} \label{def:lc_lp_cases}
	The system $S[R_1,R_2]$ is called \emph{limit point} at $b$ if $\dim\sN_{\lambda} (T_{\max}) = 1$, or \emph{limit circle} at $b$ if $\dim\sN_{\lambda} (T_{\max}) = 2$.
\end{definition}
\begin{remark}
A matrix version of integral equation equivalent to the integral system $S[R_1,R_2]$ with $R_1(x)\equiv x$ and $R_2(x)$ continuous was considered in~\cite{ArDy12}. Such equation can be reduced to a canonical differential system, see~\cite[Section 2.2]{ArDy12}. Condition of definiteness of general matrix canonical differential system was found in \cite{LM03}. In the scalar case this condition coincides with Assumption~\ref{assum:surj}.
\end{remark}

\section{Integral systems in the limit circle case}
\subsection{The fundamental matrix of the system $S[R_1,R_2]$}
We will start with some general properties of the fundamental matrix of the system $S[R_1,R_2]$.
\begin{lemma}\label{lem:2.4}
Let $U(x,\lambda)$ be the fundamental matrix function of the system $S[R_1,R_2]$
\begin{equation}\label{eq:Fund2}
		U(x,\lambda) = \begin{bmatrix} c_1(x,\lambda) & s_1(x,\lambda) \\ c_2(x,\lambda) & s_2(x,\lambda) \end{bmatrix}.
	\end{equation}
Then:
\begin{enumerate}
	\item [(i)] The following identity holds
	\begin{equation}\label{eq:Kernel_U}
		J-U(x,\mu)^*JU(x,\lambda) = -(\lambda-\conj{\mu})\int_0^x
		\begin{bmatrix} c_1(t,\conj{\mu})  \\ s_1(t,\conj{\mu})\end{bmatrix}
		\begin{bmatrix} c_1(t,\lambda) & s_1(t,\lambda) \end{bmatrix} dR_2(t).
	\end{equation}
	\item [(ii)] For every $x\in [0,b)$, $U(x,\lambda)$ is entire in $\lambda$.
	\item [(iii)] The entries of $U(x,\lambda)$ are nonnegative for $x\in [0,b)$, $\lambda\in\dR_-$.
If, in addition, the interval $(0,x)$ contains growth points of $R_1$ and $R_2$, and

\begin{equation}\label{eq:supp}
	a=\inf \supp dR_2, \quad
	a_1=\inf (\supp dR_1\cap(a,b)),
\end{equation}
then
\begin{equation}\label{eq:c_12}
	\lim_{\lambda\to-\infty} c_1(x,\lambda) = +\infty, \quad x\in(a_1,b);\quad
	\lim_{\lambda\to-\infty} c_2(x,\lambda) = +\infty, \quad x\in(a,b);
\end{equation}
\begin{equation}\label{eq:s_12}
	\lim_{\lambda\to-\infty} s_1(x,\lambda) = +\infty, \quad x\in(a_1,b); \quad
	\lim_{\lambda\to-\infty} s_2(x,\lambda) = +\infty, \quad x\in(a,b).
\end{equation}

	\item [(iv)] If $\lambda\in\dR_-$ then
\begin{equation}\label{eq:sc1}
	\frac{s_1(x,\lambda)}{c_1(x,\lambda)}< \frac{s_2(x,\lambda)}{c_2(x,\lambda)},\quad x\in (a,b),
\end{equation}
the function $\frac{s_1(x,\lambda)}{c_1(x,\lambda)}$ is increasing on $[0,b)$
and the function $\frac{s_2(x,\lambda)}{c_2(x,\lambda)}$ is decreasing on $(a,b)$.
\end{enumerate}
\end{lemma}
\begin{proof}
\textbf{1.}
By \eqref{eq:G1}, for the triples
$(c_1(\cdot,\lambda), c_2(\cdot,\lambda), \lambda c_1(\cdot,\lambda))$
and
$(c_1(\cdot,\mu), c_2(\cdot,\mu), \mu c_1(\cdot,\mu))$
belonging to $\cT$ one obtains
\begin{equation}\label{eq:V11}
	(\lambda-\conj{\mu}) \int_0^x c_1(t,\lambda) c_1(t,\conj{\mu})\, dR_2 =
	c_1(x,\lambda) c_2(x,\conj{\mu}) - c_2(x,\lambda) c_1(x,\conj{\mu}).
\end{equation}
Similarly, for
$(s_1(\cdot,\lambda),s_2(\cdot,\lambda), \lambda s_1(\cdot,\lambda))$
and
$(s_1(\cdot,\mu),s_2(\cdot,\mu), \mu s_1(\cdot,\mu))$
one obtains
\begin{equation}\label{eq:V22}
	(\lambda-\conj{\mu}) \int_0^x s_1(t,\lambda) s_1(t,\conj{\mu})\, dR_2 =
	s_1(x,\lambda) s_2(x,\conj{\mu}) - s_2(x,\lambda) s_1(x,\conj{\mu}).
\end{equation}
And finally for
$(c_1(\cdot,\lambda),c_2(\cdot,\lambda), \lambda c_1(\cdot,\lambda))$
and
$(s_1(\cdot,\mu),s_2(\cdot,\mu), \mu s_1(\cdot,\mu))$
one obtains
\begin{equation}\label{eq:V12}
	(\lambda-\conj{\mu}) \int_0^x c_1(t,\lambda) s_1(t,\conj{\mu})\, dR_2 =
	c_1(x,\lambda) s_2(x,\conj{\mu}) - c_2(x,\lambda) s_1(x,\conj{\mu}) - 1.
\end{equation}
The statement (i) is implied by~\eqref{eq:V11}--\eqref{eq:V12}.

\textbf{2.}
It follows from~\eqref{eq:Kernel_U} that
\begin{equation*}
	U(x,\mu)^* = JU(x,\conj{\mu})^{-1} J^T.
\end{equation*}

Therefore,
\begin{equation*}
	\frac{U(x,\lambda) - U(x,\conj{\mu})}{\lambda-\conj{\mu}}=
	U(x,\conj{\mu}) J^T\int_0^x
	\begin{bmatrix} c_1(t,\conj{\mu})  \\ s_1(t,\conj{\mu})\end{bmatrix}
		\begin{bmatrix} c_1(t,\lambda) & s_1(t,\lambda) \end{bmatrix} dR_2(t),
\end{equation*}
hence $U(x,\lambda)$ is holomorphic on $\dC$which shows (ii).

\textbf{3.} To show (iii), expanding $c_1(x,\lambda)$ and $c_2(x,\lambda)$ in series in $\lambda$
\begin{equation*}
c_1(x,\lambda)=1-\lambda\varphi_1(x)+\lambda^2\varphi_2(x)+\dots,\quad
c_2(x,\lambda)=-\lambda\psi_1(x)+\lambda^2\psi_2(x)+\dots
\end{equation*}
one obtains from~\eqref{eq:IntSys} that
\begin{equation}\label{eq:phi_1}
	\psi_1(x)=R_2(x), \quad \varphi_1(x)=\int_0^xR_2(t)\,dR_1(t)
\end{equation}
\begin{equation}\label{eq:phi_n}
	\psi_n(x)=\int_0^x\varphi_{n-1}(t)\,dR_2(t), \quad \varphi_n(x)=\int_0^xdR_1(t)\int_0^t\varphi_{n-1}(s)\,dR_2(s),\quad n\in\dN,\,n\ge 2.
\end{equation}
This implies that $\varphi_n(x)\ge 0$, $\psi_n(x)\ge 0$ for $n\in \dN$ and hence
\begin{equation*}
	c_1(x,\lambda)\ge 0, \quad c_2(x,\lambda)\ge 0
	\quad \text{for} \quad x\in [0,b),\ \lambda\in\dR_-.
\end{equation*}
Moreover, it follows from~\eqref{eq:phi_1} that
\begin{equation}\label{eq:phi_ineq}
	c_1(x,\lambda)\ge 1+|\lambda|\int_0^xR_2(t)\,dR_1(t), \quad
	c_2(x,\lambda)\ge |\lambda|R_2(x).
\end{equation}
Therefore, the relations \eqref{eq:c_12} hold since
\begin{equation*}
	\int_0^x R_2(t)\,dR_1(t) > 0 \ \text{for}\ x\in(a_1,b)
	\quad \text{and} \quad
	R_2(x)>0 \ \text{for}\ x\in(a,b).
\end{equation*}

The proof of \eqref{eq:s_12} is similar.

\textbf{4.}
The identity \eqref{eq:Liouv} yields
\begin{equation}\label{eq:cs12Id}
	\frac{s_2(x,\lambda)}{c_2(x,\lambda)}-\frac{s_1(x,\lambda)}{c_1(x,\lambda)}=\frac{1}{c_1(x,\lambda)c_2(x,\lambda)}
\end{equation}
This proves the inequality~\eqref{eq:sc1}.

It follows from \eqref{eq:IntSys}, \eqref{eq:in-by-parts}, \eqref{eq:quotient-rule}, and \eqref{eq:Liouv+} that
\begin{equation*}
	d\left(\frac{s_1(x,\lambda)}{c_1(x,\lambda)}\right) =
	\frac{c_{1+}(x,\lambda)s_2(x,\lambda)-c_2(x,\lambda)s_{1+}(x,\lambda)}{c_1(x,\lambda) c_{1+}(x,\lambda)} dR_1(x) =
	\frac{1}{c_1(x,\lambda) c_{1+}(x,\lambda)}\,dR_1(x)
\end{equation*}
and hence
\begin{equation}\label{eq:c1s1}
	\frac{s_1(x,\lambda)}{c_1(x,\lambda)}=\int_0^x\frac{1}{c_1(t,\lambda) c_{1+}(t,\lambda)}\,dR_1(t).
\end{equation}
Since $c_1(x,\lambda)$, $ c_{1+}(x,\lambda)>0$ for $\lambda\in\dR_-$ and $x\in[0,b)$, the function $\frac{s_1(x,\lambda)}{c_1(x,\lambda)}$ is increasing on $[0,b)$.

Similarly, by~\eqref{eq:IntSys}, \eqref{eq:in-by-parts}, \eqref{eq:quotient-rule}, and \eqref{eq:Liouv++}
\begin{equation}\label{eq:cs2}
	d\left(\frac{c_2(x,\lambda)}{s_2(x,\lambda)}\right) =
\frac{-\lambda}{s_2(x,\lambda) s_{2+}(x,\lambda)}\,dR_2(x), \quad x\in[0,b)
\end{equation}
and hence the function $\frac{c_2(x,\lambda)}{s_2(x,\lambda)}$ is increasing on $[0,b)$.
This proves (iv).
Notice, that the function $\frac{s_2(x,\lambda)}{c_2(x,\lambda)}$ is not defined on $[0,a]$.

\end{proof}

\subsection{The Evans-Everitt condition in the limit circle case}
\begin{proposition} \label{prop:1}
	The system $S[R_1,R_2]$
is limit circle at $b$ if and only if $1, R_1\in \cL^2(R_2)$.
\end{proposition}
\begin{proof}
	Using the well-known procedure from~\cite[Theorem 5.6.1]{Atk64} (see also \cite[Theorem 4.5]{Str19}) one can show that $S[R_1,R_2]$ is limit circle at $b$
if and only if $c_1(x,0)$ and $s_1(x,0) $ belong to $\cL^2(R_2)$.
	Substitution of $\lambda=0$ to~\eqref{eq:IntSys} immediately provides $c_2(x,0) = 0$, $s_2(x,0) = 1 $ and hence $c_1(x,0) = 1$, $s_1(x,0) = R_1(x)$.
\end{proof}

If the system $S[R_1,R_2]$
is regular at $b$, then the following limits exist:
\begin{equation}\label{eq:lim_1}
	c_1(b,\lambda)=\lim_{t\to b}c_1(t,\lambda),\quad s_1(b,\lambda)=\lim_{t\to b}s_1(t,\lambda),
\end{equation}
\begin{equation}\label{eq:lim_2}
	c_2(b,\lambda)=\lim_{t\to b}c_2(t,\lambda),\quad s_2(b,\lambda)=\lim_{t\to b}s_2(t,\lambda).
\end{equation}

Assume now that the system $S[R_1,R_2]$
is in the limit circle at $b$.
One can check (see \cite[Section 10.7]{KacK68}, \cite[Theorem 3.8]{Str18}) that for every element $\bm{u} = \begin{bmatrix} \pi u_1 \\ \pi f \end{bmatrix} \in T_{\max}$ the limit
\begin{equation} \label{eq:f2d}
	u_2(b) = u_2(0) - \int_0^b f\,dR_2
\end{equation}
exists and is well defined.
Therefore, the limits~\eqref{eq:lim_2} exist.

\begin{lemma}\label{lem:psi}
	Let the system $S[R_1,R_2]$
be limit circle at $b$.
	Then for every $\bm{u} = \begin{bmatrix} \pi u_1 \\ \pi f \end{bmatrix} \in T_N^*$ one has $u_2\in \cL^2(R_1)$ and the following two equalities hold:
	\begin{equation} \label{eq:2.26}
		\lim_{x\to b} u_1(x) = u_1(0) + (f,R_1),
	\end{equation}
	\begin{equation} \label{eq:LC-SLP}
		\lim_{x\to b} u_1(x) u_2(x) = 0.
	\end{equation}
	If $\bm{u} = \begin{bmatrix} \pi u_1 \\ \pi f \end{bmatrix} \in T_{\max}$ and the endpoint $b$ is singular then $\bm{u} = \begin{bmatrix} \pi u_1 \\ \pi f \end{bmatrix} \in T_N^*$ provided \eqref{eq:LC-SLP} holds.
\end{lemma}
\begin{proof}
	Let $\bm{u} = \begin{bmatrix} \pi u_1 \\ \pi f \end{bmatrix} \in T_N^*$.
	Applying the integration-by-parts formula~\eqref{eq:in-by-parts} to the first line of~\eqref{eq:sys_lin_rel} one gets
	\begin{equation} \label{eq:u1_ibp}
		u_1(x) = u_1(0) + u_2(x) R_1(x) + \int_0^x R_1(t) f(t)\, dR_2(t).
	\end{equation}
	We recall that in the limit circle case $1, R_1 \in \cL^2(R_2)$ and $f \in \cL^2(R_2)$ by the assumption of the lemma.
	The condition $u_2(b)=0$ provides $u_2(x) = \int_x^b f\, dR_2$ and hence~\eqref{eq:u1_ibp} can be rewritten as
	\begin{equation} \label{eq:l1}
		u_1(x) = u_1(0) + (f,R_1) - \int_x^b (R_1(t)-R_1(x)) f(t)\, dR_2(t).
	\end{equation}
	 Note the following estimation:
	\begin{equation} \label{eq:l2}
		\begin{split}
			\left| \int_x^b (R_1(t)-R_1(x)) f(t)\, dR_2(t) \right| &\le
			\int_x^b (R_1(t)-R_1(x)) |f(t)|\, dR_2(t) \\
			&\le \int_x^b R_1 |f|\, dR_2 \to 0\quad \text{as}\quad x \to b.
		\end{split}
	\end{equation}
	Now~\eqref{eq:2.26} follows from \eqref{eq:l1} and \eqref{eq:l2}, and~\eqref{eq:LC-SLP} finally follows from~\eqref{eq:2.26}.

	The claim $u_2\in \cL^2(R_1)$ for $\bm{u} = \begin{bmatrix} \pi u_1 \\ \pi f \end{bmatrix} \in T_N^*$ follows from~\eqref{eq:2.26} and the first Green's identity~\eqref{eq:G1}
	\begin{equation}\label{eq:Green1LC}
		\begin{split}
			\int_0^b f(t) \conj{u_1(t)}\,dR_2(t)
			&= \int_0^b |u_2|^2 dR_1(t) - \lim_{x\to b} u_2(x) \conj{u_1(x)} + u_2(0) \conj{u_1(0)}\\
			&= \int_0^b |u_2|^2 dR_1(t) + u_2(0) \conj{u_1(0)}.
		\end{split}
	\end{equation}

	Now assume that the endpoint $b$ is singular and $\bm{u} = \begin{bmatrix} \pi u_1 \\ \pi f \end{bmatrix} \in T_{\max}$.
	From~\eqref{eq:f2d} we have $u_2(b)=a$ where $a \in \nC$.
	In the limit circle case the singular endpoint $b$ implies $R_1(b)=\infty$.
	If $a\neq0$ then from~\eqref{eq:sys_lin_rel} we get $u_1(b)=\pm\infty$ and hence~\eqref{eq:LC-SLP} does not hold.
\end{proof}
\begin{remark}
The condition \eqref{eq:LC-SLP} for Sturm-Liouville operators
in the limit circle case was introduced and studied by Evans and Everitt in~\cite{EvEver}.
We will call it the Evans--Everitt condition.
\end{remark}

\subsection{Boundary triples for integral systems in the limit circle case}
\begin{definition}[see~\cite{Ben89,Kost13}] \label{def:5.1}
	The function $m(\lambda,b,\infty)$ from~\eqref{eTWeyl} for which the solution
	\begin{equation} \label{eTWeyl2}
		\psi^N(t,\lambda)=s(t,\lambda)-m(\lambda,b,\infty){c}(t,\lambda), \qquad t \in I,
	\end{equation}
	satisfies the condition
	\begin{equation} \label{eq:m_functLC}
		\psi^N_2(b,\lambda) = 0,
	\end{equation}
	is called the \emph{Neumann $m$-function} of the system $S[R_1,R_2]$
on $I$ subject to the boundary condition~\eqref{eq:m_functLC}.
\end{definition}

It follows from~\eqref{eTWeyl} and the condition $\psi_2^N(b,\lambda)=0$ that
$s_2(b,\lambda)-m(\lambda,b,\infty)c_2(b,\lambda)=0$ which proves the formula
\begin{equation}\label{eq:WFA_N2}
 m(\lambda,b,\infty)=
 \frac{s_2(b,\lambda)}{c_2(b,\lambda)}.
\end{equation}

We will show below that the function $ m(\lambda,b,\infty)$ is a Weyl function of a one-dimensional symmetric extension $T_N$ of the linear relation $T_{\min}$ defined by
\begin{equation}\label{eq:A_N}
	T_{N} = \left\{
	\bm{u} = \begin{bmatrix} \pi u_1 \\ \pi f \end{bmatrix} \colon
	(u_1,u_2,f) \in \cT,\ u_1(0)=u_2(0)=u_2(b)=0 \right\}.
\end{equation}
As follows from~\eqref{eq:G2} the adjoint linear relation $T_N^*$ is of the form
\begin{equation}
	T_N^* = \left\{
	\bm{u}=\begin{bmatrix} \pi u_1 \\ \pi f \end{bmatrix} \colon
	(u_1,u_2,f) \in \cT:\, u_2(b)=0 \right\}.
\end{equation}

\begin{proposition}\label{prop:BT_IS_N}
	Let the system $S[R_1,R_2]$ be singular and limit circle at $b$, let $T_N$ be defined by \eqref{eq:A_N}, and let $m(\lambda,b,\infty)$ be the Neumann $m$-function of the system $S[R_1,R_2]$ given by~\eqref{eq:WFA_N2}.
	Then:
\begin{enumerate}
\item [(i)] $T_N$ is a symmetric nonnegative linear relation in $L^2(R_2)$ with deficiency indices $(1,1)$.
\item [(ii)] The triple $\Pi^N = (\dC,\Gamma_0^N,\Gamma_1^N)$, where
\begin{equation}\label{eq:GammaA_N}
	\Gamma_0^N \bm{u} = u_2(0), \quad \Gamma_1^N \bm{u} = -u_1(0), \quad \bm{u} \in T_N^*,
\end{equation}
is a boundary triple for $T_N^*$.
\item[(iii)] The Weyl function $m_N(\lambda)$ of $T_N$ corresponding to the boundary triple $\Pi^N$ coincides with the Neumann $m$-function $m(\lambda,b,\infty)$.
\item[(iv)] The Weyl function $m_N(\lambda)$ of $T_N$ coincides with the principal Titchmarsh-Weyl coefficient $q(\lambda)$ of the system $S[R_1,R_2]$,
belongs to the Stieltjes class $\cS$, and
\begin{equation}\label{eq:lim}
	\lim_{\lambda\to-\infty}m_N(\lambda)=R_{1+}(a),
\end{equation}
where $a=\inf\supp dR_2$.
\item[(v)] The Weyl function $m_N(\lambda)$ of $T_N$ admits the representation
	\begin{equation}\label{eq:Pol_m}
		m_N(\lambda)=-\frac{1}{R_2(b)\cdot\lambda}+\wt m(\lambda);
	\end{equation}
where $\wt m$ is a function from $\cS$ such that $\lim_{y\to 0}y\wt m(iy)=0$.
\end{enumerate}
\end{proposition}
\begin{proof}
\textbf{1.} To show ${\textrm (i)}$ and ${\textrm (ii)}$,
let the tuples $(u_1,u_2,f)$ and $(v_1,v_2,g)$ satisfy the system~\eqref{eq:sys_lin_rel} and assume that $u_2(b)=v_2(b)=0$, i.e.
$\bm{u}, \bm{v} \in T_N^*$.
Let $\mu\in\dR$. By
formula~\eqref{eq:Liouv} at least one of the values $c_2(b,\mu)$ and $s_2(b,\mu)$ is not equal to 0. Assume that $c_1(b,\mu)\ne 0$ and let us set $c(x) \coloneqq c(x,\mu)$.
Due to the identity
\begin{equation}\label{eq:Wron_Id}
	[u,v]_b = c_2(b,\mu)^{-1}
	\left\{ [u(b),c(b,\mu)] \conj{v_2(b)} - u_2(b) [\conj{v(b)},c(b,\mu)] \right\}
\end{equation}
the second Green's identity \eqref{eq:G2} is of the form
\begin{equation}\label{eq:GreenN}
  \int_0^b (f\conj{v_1} - u_1\conj{g})\, dR_2(t) =
  [u,\conj{v}]_b - [u,\conj{v}]_0
  =u_2(0)\conj{v_1(0)}-u_1(0)\conj{v_2(0)}.
\end{equation}
By Definition~\ref{def:btriple}
 the boundary triple for $T_N^*$ can be taken as $\Pi^N = (\dC,\Gamma_0^N,\Gamma_1^N)$, with
$\Gamma_0^N, \Gamma_1^N$ given in \eqref{eq:GammaA_N}.

It follows from the first Green's identity \eqref{eq:Green1LC} and Lemma~\ref{lem:psi} that for every $(\pi u_1,\pi f)^T\in T_N$
\begin{equation} \label{eq:Green1_}
  \int_0^b f(t)\conj{u_1(t)}\,dR_2(t)=\int_0^b |u_2|^2dR_1(t)\ge 0.
\end{equation}

\textbf{2.} Now ${\textrm (iii)}$ is shown.
The defect subspace $\sN_\lambda(T_N^*)$ is spanned by the function $\psi_1^N(\cdot,\lambda)$,
where $\psi^N(t,\lambda)$ is the Weyl solution from \eqref{eTWeyl2} corresponding to the Neumann $m$-function $m(\lambda,b,\infty)$. Denote $\bm{u}^N(t,\lambda)=(\psi_1^N(\cdot,\lambda),\lambda\psi_1^N(\cdot,\lambda))^T\in\wh \sN_\lambda(T_N^*)$.
Using the formulas~\eqref{eTWeyl2} and~\eqref{eq:GammaA_N} one obtains
\begin{equation*}
\Gamma_1^N\bm{u}^N(\cdot,\lambda)=-\psi_1^N(0,\lambda)=m(\lambda,b,\infty),\quad \Gamma_0^N\bm{u}^N(\cdot,\lambda)=\psi^N_2(0,\lambda)=1
\end{equation*}
and hence by~\eqref{def:Weyl} the Weyl function $m_N(\lambda)$ is of the form
\begin{equation}
  m_N(\lambda)=\frac{\Gamma_1^N\bm{u}^N(\cdot,\lambda)}{\Gamma_0^N\bm{u}^N(\cdot,\lambda)}
  =m(\lambda,b,\infty).
\end{equation}
Therefore, the Weyl function $m_N(\lambda)$ coincides with
the Neumann $m$-function $m(\lambda,b,\infty)$.

\textbf{3.}
The inclusion $m_N\in\cS$ follows from Lemma~\ref{lem:2.4}, since the functions $s_2(x,\lambda)$ and $c_2(x,\lambda)$ are positive for $\lambda<0$ and the function $ m_N(\lambda)$ admits a holomorphic nonnegative continuation on $\dR_-$.

Let $a=\inf\supp R_2$ and $a_1=\inf(\supp R_1\cap(a,b))$. Then by Assumption~\ref{assum:surj} $a_1<b$ and due to~\eqref{eq:IntSys} and Lemma~\ref{lem:2.4} (iii)
\begin{equation*}
	c_1(x,\lambda)\equiv 1\ \text{for}\ x\le a_1
	\quad \text{and} \quad
	\lim_{\lambda\to-\infty} c_1(x,\lambda)=\infty\ \text{for}\ x\ge a_1.
\end{equation*}
Now we must consider two cases. In case if $c_1(\cdot,\lambda)$ has a jump at point $a_1$, which is only possible if $a_1>a$, we get
\begin{equation}
	\frac{1}{c_1(x,\lambda) c_{1+}(x,\lambda)} \to \chi_{[0,a_1)}(x)
	\quad \text{as} \quad \lambda \to -\infty
\end{equation}
and hence by the Lebesgue bounded convergence theorem one obtains from~\eqref{eq:c1s1}
\begin{equation}\label{eq:c1s1_Lim}
	\lim_{\lambda\to-\infty} \frac{s_1(x,\lambda)}{c_1(x,\lambda)} =
	\int_0^x \frac{1}{c_1(x,\lambda) c_{1+}(x,\lambda)}\, dR_1 =
	\int_{[0,a_1)}\,dR_1 = R_1(a_1) = R_{1+}(a).
\end{equation}
The last equality in~\eqref{eq:c1s1_Lim} follows from $a_1>a$ and the definition of the points $a$, $a_1$.

In case if $c_1(\cdot,\lambda)$ has no jump at point the $a_1$, which is possible either if $a_1=a$ or $a_1>a$ and $R_1$ has no jump at $a_1$,
we get
\begin{equation}
	\frac{1}{c_1(x,\lambda) c_{1+}(x,\lambda)} \to \chi_{[0,a_1]}(x)
	\quad \text{as} \quad \lambda \to -\infty
\end{equation}
and similarly to~\eqref{eq:c1s1_Lim}
\begin{equation}
	\lim_{\lambda\to-\infty} \frac{s_1(x,\lambda)}{c_1(x,\lambda)} =
	R_{1+}(a_1) = R_{1+}(a).
\end{equation}

Since $R_1(b)+R_2(b)=+\infty$ it follows from \eqref{eq:phi_ineq} that $\lim_{x\to b} c_1(x,\lambda) c_2(x,\lambda) = +\infty$ for all $\lambda\in\dR_-$ and hence it follows from~\eqref{eq:cs12Id} that
\begin{equation*}
	q(\lambda)=\lim_{x\to b}\frac{s_1(x,\lambda)}{c_1(x,\lambda)} =
	\lim_{x\to b} \frac{s_2(x,\lambda)}{c_2(x,\lambda)} = m_N(\lambda),
	\quad \lambda\in\dR_-.
\end{equation*}
Since $q$ and $m_N$ are holomorphic on $\dC\setminus\dR_+$ this proves that $q(\lambda)\equiv m_N(\lambda)$, and (iv) is shown.

\textbf{4.} Now we prove ${\textrm (v)}$.
It follows from~\eqref{eq:IntSys} that
\begin{equation}\label{eq:s2}
  s_2(x,\lambda)=1-\lambda\int_0^x s_1(x,\lambda)\,dR_2(t), \quad
  c_2(x,\lambda)=-\lambda\int_0^x c_1(x,\lambda)\,dR_2(t)
\end{equation}
and by~\eqref{eq:WFA_N2} that
\begin{equation}
  m_N(\lambda)=\frac{1-\lambda\int_0^b s_1(x,\lambda)\,dR_2(t)}{-\lambda\int_0^b c_1(x,\lambda)\,dR_2(t)},\quad \lambda\in\dC\setminus\dR.
\end{equation}
Moreover, for $\lambda<0$ the functions $s_1(x,\lambda)$ and $c_1(x,\lambda)$ are positive and increasing on $(0,b)$ and $c_2(0,\lambda)=1$, hence
\begin{equation}
	\int_0^b c_1(x,\lambda)\,dR_2(t)>R_2(b), \quad\int_0^b s_1(x,\lambda)\,dR_2(t)>0.
\end{equation}
Since $c_1(x,\lambda)\to c_1(x,0)\equiv 1$ and $s_1(x,\lambda)\to s_1(x,0)=R_1(x)$ as $\lambda\to 0-$ and these convergences are monotone and uniform on $[0,b]$ one finds that
\begin{equation}
	\int_0^b c_1(t,\lambda)\,dR_2(t)\to R_2(b),\quad \int_0^b s_1(x,\lambda)\,dR_2(t)\to \int_0^b R_1(t)\,dR_2(t), \quad\text{as}\quad \lambda\to 0-.
\end{equation}
Therefore,
\begin{equation}
 \lambda m_N(\lambda)\to -\frac{1}{R_2(b)},\quad \text{as}\quad\lambda\to 0-
\end{equation}
and thus $m_N(\lambda)$ admits the representation~\eqref{eq:Pol_m}.
\end{proof}

\subsection{Integral systems in the regular case}
\label{subsec:q-r}

Assume that the system $S[R_1,R_2]$
is regular at~$b$, i.e. $R_1(b)+R_2(b)<\infty$.
Then for every tuple $(u_1,u_2,f) \in \cT$ it follows from~\eqref{eq:f2d} that the function $u_2$ is bounded and hence the limit
\begin{equation} \label{eq:f1d}
	u_1(b) = u_1(0) + \int_0^b u_2\,dR_1
\end{equation}
exists and well defined.
Therefore, the limits~\eqref{eq:lim_1} exist.

\begin{definition}\label{def:5.1D}{\textrm (see~\cite{Ben89,Kost13})}
The function $m(\lambda,b,0)$ for which the solution
\begin{equation} \label{eTWeylND}
\psi^{ND}(t,\lambda)=s(t,\lambda)-m(\lambda,b,0){c}(t,\lambda), \qquad t
\in I,
\end{equation}
 satisfies the condition
\begin{equation}\label{eq:m_functLCD}
\psi^{ND}_1(b,\lambda)=0
\end{equation}
is called the \emph{Neumann $m$-function} of the system $S[R_1,R_2]$
on $I$ subject to the boundary condition~\eqref{eq:m_functLCD}.
\end{definition}
It follows from~\eqref{eTWeyl} and the condition $\psi^{ND}_1(b,\lambda)=0$ that
$s_1(b,\lambda)-m(\lambda,b,0)c_1(b,\lambda)=0$ which yields the formula
\begin{equation}\label{eq:WFA_N1}
  m(\lambda,b,0)=
 \frac{s_1(b,\lambda)}{c_1(b,\lambda)}
\end{equation}
and hence the Neumann $m$-function $ m(\lambda,b,0)$ coincides with the principal Titchmarsh-Weyl coefficient $q(\lambda)$ of the system $S[R_1,R_2]$.

Let $T_D$ be a symmetric extension of the linear relation $T_{\min}$ defined by
\begin{equation}\label{eq:A_D}
	T_{D} = \left\{ \bm{u} = \begin{bmatrix} \pi u_1 \\ \pi f \end{bmatrix} \colon (u_1,u_2,f) \in \cT,\ u_1(0) = u_2(0) = u_1(b) = 0 \right\}.
\end{equation}
As follows from~\eqref{eq:G2} the adjoint linear relation $T_D^*$ is of the form
 \begin{equation}
 	T_D^* = \left\{\bm{u} = \begin{bmatrix} \pi u_1 \\ \pi f \end{bmatrix} \colon (u_1,u_2,f) \in \cT:\, u_1(b)=0 \right\}.
 \end{equation}

\begin{proposition}[cf. \cite{Str18}] \label{prop:BT_IS_D}
	Let the system $S[R_1,R_2]$ be regular at $b$, and let $T_D$ be defined by~\eqref{eq:A_D}.
	Then:
\begin{enumerate}
\item[(i)]  $T_D$ is a symmetric nonnegative linear relation in
$L^2(R_2)$ with deficiency indices $(1,1)$ and $u_2\in\ L^2(R_1)$ for all $\bm{u} = \begin{bmatrix} \pi u_1 \\ \pi f \end{bmatrix}\in T_D^*$;
    \item[(ii)]  the triple $\Pi^{ND} = (\dC,\Gamma_0^{ND},\Gamma_1^{ND})$, where
\begin{equation}\label{eq:GammaA_ND}
  \Gamma_0^{ND} \bm{u} = u_2(0), \quad \Gamma_1^{ND} \bm{u} = -u_1(0), \quad \bm{u} \in T_D^*,
\end{equation}
is a boundary triple for $T_D^*$.
\item[(iii)] The Weyl function $m_{ND}(\lambda)$ of $T_D$ corresponding to the boundary triple $\Pi^{ND}$ coincides with $m(\lambda,b,0)$.
\item[(iv)] The Weyl function $m_{ND}$ of $T_D$ belongs to the Stieltjes class $\cS$ and coincides with the principal Titchmarsh-Weyl coefficient $q(\lambda)$ of the system $S[R_1,R_2]$.
\end{enumerate}
\end{proposition}
\begin{proof}
\textbf{1.} To show ${\textrm (i)}$ and ${\textrm (ii)}$,
let the tuples $(u_1,u_2,f)$ and $(v_1,v_2,g)$ satisfy the system \eqref{eq:sys_lin_rel}
and assume that $u_1(b)=v_1(b)=0$, i.e. $\bm{u}, \bm{v} \in T_D^*$.
Let $\mu\in\dR$. By the Liouville-Ostrogradskii formula~\eqref{eq:Liouv} at least one of the values $c_1(b,\mu)$ and $s_1(b,\mu)$ is not equal to 0. Assume that $c_1(b,\mu)\ne 0$  and let us set $c(x) \coloneqq c(x,\mu)$.
  Due to the identity
\begin{equation}\label{eq:Wron_IdD}
  [u,v]_b = c_1(b,\mu)^{-1}\left\{[u(b),c(b,\mu)]\conj{v_1(b)} - u_1(b)[\conj{v(b)},c_(b,\mu)]\right\}
\end{equation}
the Green's identity \eqref{eq:G2} is of the form~\eqref{eq:GreenN}.
By Definition~\ref{def:btriple} the boundary triple for $T_D^*$ can be taken as $\Pi^{ND} = (\dC,\Gamma_0^{ND},\Gamma_1^{ND})$, with
$\Gamma_0^{ND},\Gamma_1^{ND}$ given in \eqref{eq:GammaA_ND}.

It follows from the the first Green's identity \eqref{eq:G1} and Lemma~\ref{lem:psi} that for every $\bm{u} \in T_D$ the identity~\eqref{eq:Green1_} holds and thus the linear relation $T_D$ is nonnegative.

\textbf{2.} Now ${\textrm (iii)}$ is shown.
The defect subspace $\sN_\lambda(T_D)$ is spanned by the function $\psi_1^{ND}(\cdot,\lambda)$ determined by~\eqref{eTWeylND}.
Denote $\bm{u}^{ND}(t,\lambda)=(\psi_1^{ND}(\cdot,\lambda),\lambda\psi_1^{ND}(\cdot,\lambda))^T\in\wh \sN_\lambda(T_{D}^*)$.
Using the formulae~\eqref{eTWeyl} and~\eqref{eq:csK} one obtains
\begin{equation*}
	\Gamma_1^{ND}\bm{u}^{ND}(\cdot,\lambda) = -\psi_1^{ND}(0,\lambda) = m(\lambda,b,0), \quad
	\Gamma_0^{ND}\bm{u}^{ND}(\cdot,\lambda) = \psi^{ND}_2(0,\lambda) = 1
\end{equation*}

and hence the Weyl function $m_{ND}(\lambda)$ is of the form
\begin{equation}
	m_{ND}(\lambda) =
	\frac{\Gamma_1^{ND} \bm{u}^{ND}(\cdot,\lambda)}{\Gamma_0^{ND} \bm{u}^{ND}(\cdot,\lambda)} =
	m(\lambda,b,0).
\end{equation}
Therefore, the Weyl function $m_{ND}(\lambda)$ coincides with
the Neumann $m$-function $m(\lambda,b,0)$.

\textbf{3.} Finally we prove ${\textrm (iv)}$.
The inclusion $m_{ND}\in\cS$ follows from Lemma~\ref{lem:2.4}. The equality $m_{ND}(\lambda)\equiv q(\lambda)$, $\lambda\in\dC\setminus\dR$, is implied by~\eqref{eq:WFA_N1}.
\end{proof}
\begin{remark} \label{rem:Param}
The functions $R_1$ and $R_2$ are not uniquely defined by the principal Titchmarsh-Weyl coefficient of the system $S[R_1,R_2]$.
As was shown in \cite[Lemma 2.12]{Kost13}
if functions $\wt R_1(\xi)$ and $\wt R_2(\xi)$ are connected by
\begin{equation*}
	\wt R_1(\xi) = R_1(x(\xi)), \quad \wt R_2(\xi) = R_2(x(\xi)), \quad \xi\in[0,\beta].
\end{equation*}

where $x(\xi)$ is an increasing function on the interval $[0,\beta]$, such that $x(0)=0$ and $x(\beta)=b$, then the principal Titchmarsh-Weyl coefficient $\wt q$ of the system
\begin{equation}\label{eq:IntSysW}
	\wt u(\xi,\lambda) = \wt u(0,\lambda) -J \int_0^\xi
	\begin{bmatrix} \lambda d\wt R_2(\tau) & 0   \\  0 & d\wt R_1(\tau)\end{bmatrix} \wt u(\tau,\lambda),\quad \xi\in[0,\beta].
\end{equation}
 coincides with the principal Titchmarsh-Weyl coefficient $q$ of the system $S[R_1,R_2]$.

Therefore we can always assume that for regular systems $S[R_1,R_2]$
the parameter $x$ ranges over a finite interval $[0,b]$, $b<\infty$.
\end{remark}

\begin{remark}\label{rem:MP}
As is known, see~\cite[Section A13]{KacK68}, a truncated moment problem can be reduced to a regular integral system $S[R_1,R_2]$
 with
\begin{gather*}
	R_1(x) = x, \quad R_2(x) = \sum_{j=0}^{n-1} m_j H (x-x_j), 	\quad x\in[0,x_n],\\
	x_j=\sum_{j=1}^j l_i, \quad m_{j-1}, l_j>0, \quad 1\le j\le n.
\end{gather*}
where $H(x)$ is the Heaviside function.
 The corresponding monodromy matrix $U(x_n,\lambda)$ is of the form
\begin{equation}
	U(x_n,\lambda) =\prod_{j=1}^n U_{x_{j-1}}(x_j,\lambda),\quad \text{where}\quad
	U_{x_{j-1}}(x_j,\lambda) =
    \begin{bmatrix}
		1 - \lambda l_j m_{j-1} & l_j \\
		-\lambda  m_{j-1}     & 1
	\end{bmatrix}.
\end{equation}
The system $S[R_1,R_2]$
satisfies Assumption~\ref{assum:surj} if $n>1$.
If $n=1$ then $R_2(x)=H(x)$, $x\in[0,l_1]$, $L^2(R_2)=\dC$, the system $S[R_1,R_2]$ is of the form
\begin{equation}\label{eq:1_dim}
	u_1(x) = u_1(0) + x u_2(x), \quad
	u_2(x) = u_2(0) - \lambda u_1(0) m_0, \quad
	x\in(0,l_1]
\end{equation}
and does not satisfy the Assumption~\ref{assum:surj}. However, in this case one can still introduce a boundary triple $(\dC,\Gamma_0,\Gamma_1)$ for $T_{\max}=\dC\times\dC$ by
\begin{equation}\label{eq:BT_1dim}
	\Gamma_0\bm{u}=u_1(0),\quad \Gamma_1\bm{u}=f(0), \quad
	\bm{u} = \begin{bmatrix} u_1 \\ f \end{bmatrix} \in T_{\max}
\end{equation}
and the corresponding Weyl function is $m(\lambda)=m_0\lambda$.

The system $S[\wt R_1,\wt R_2]$ with $\wt R_1(x)=l_1H(x-1)$, $\wt R_2(x)=m_0 H(x)$, $x\in[0,2]$ is equivalent to the system $S[R_1,R_2]$
in the sense that its Weyl function corresponding to the boundary triple~\eqref{eq:BT_1dim} coincides with $m(\lambda)=m_0\lambda$ and
the monodromy matrix $\wt U(2,\lambda)$ of this system coincides with $U(l_1,\lambda)$. The advantage of system $S[\wt R_1,\wt R_2]$ is that the elementary factors of $\wt U(2,\lambda)$ from its factorization
\begin{equation*}
	\wt U(2,\lambda) = U^{(1)}(\lambda)U^{(0)}(\lambda), \quad
	U^{(1)}(\lambda) =
	\begin{pmatrix}
		1 & l_1 \\
		0 & 1
	\end{pmatrix}, \quad
	U^{(0)}(\lambda) =
	\begin{pmatrix}
    	1 & 0 \\
    	-\lambda m_0 & 1
	\end{pmatrix}
\end{equation*}
can be also treated as monodromy matrices of systems $S[0,\wt R_2]$ on the interval $[0,1]$ and $S[\wt R_1,0]$ on $[1,2]$, respectively.
\end{remark}

\section{Integral systems in the limit point case}
\subsection{The strong limit point condition}
The next lemma is an analog of one result in~\cite[Lemma]{Ever76} in the case of integral systems.

\begin{lemma} \label{lem:log}
	Let $f$ be a (not necessarily strictly) monotone function on $[b_0,b)$ such that either $f(x) \to \pm\infty$ or $f(x) \to 0$ as $x\to b$ and let $f(x)\neq0$ on $[b_0,b)$. Then
	\begin{equation} \label{eq:intdf_f}
		\lim_{x\to b} \int_{b_0}^x df/f = \pm \infty.
	\end{equation}
\end{lemma}
\begin{proof}
	We will prove the lemma in the case $f>0$, $f\to 0$. The proof in the other cases is similar.
	Let $D_f $ be the set of the points of discontinuity of $f$ on $[b_0,b)$.
	One can write
	\begin{equation} \label{eq:log3.1}
		\int\limits_{[b_0,x)} \frac{df}{f} = \int\limits_{[b_0,x)\setminus D_f} \frac{df}{f} + \int\limits_{[b_0,x)\cap D_f} \frac{df}{f}.
	\end{equation}
	Notice that both the integrals on the right hand side of~\eqref{eq:log3.1} are negative, therefore if one of them diverges (as $x\to b$) then the assertion of the lemma holds.

	Let $D_f = \{x_n\}_{n=0}^{\infty}$. Consider the following inequality
	\begin{equation} \label{eq:log3.3}
		\frac{f_+(x_n) - f_-(x_n)}{f(x_n)} \le \frac{f_+(x_n) - f_-(x_n)}{f_-(x_n)} = \frac{f_+(x_n)}{f_-(x_n)} - 1 < 0
	\end{equation}
	and the associated series
	\begin{equation} \label{eq:log3.6}
		\sum_{n=0}^{\infty} \left( \frac{f_+(x_n)}{f_-(x_n)} - 1 \right).
	\end{equation}
	If series~\eqref{eq:log3.6} diverges then the following integral
	\begin{equation} \label{eq:log3.4}
		\int\limits_{[b_0,b)\cap D_f} \frac{df}{f} = \sum_{x_n\in D_f} \frac{f_+(x_n) - f_-(x_n)}{f(x_n)}
	\end{equation}
	diverges as well, so the assertion of the lemma holds immediately.

	Assume now that series~\eqref{eq:log3.6} converges and denote $a_n \coloneqq 1 - f_+(x_n)/f_-(x_n)$. Notice that the measure $d\log(f)$ is absolutely continuous with respect to $df$ and therefore there exists the Radon-Nikodym derivative $d\log(f)/df \in L^1(df)$ which has a representative (see \cite[5.3, formula (3.5)]{BerUsShe})
	\begin{equation} \label{eq:log3.radon-nikodym}
		\frac{d\log(f)}{df} = \left\{
		\begin{array}{ll}
			1/f(x), & x \in [b_0,b)\setminus D_f,\\
			(\log f_+(x) - \log f_-(x))/(f_+(x)-f_-(x)), & x \in D_f.
		\end{array} \right.
	\end{equation}
	Now we get by the Radon-Nikodym theorem
	\begin{equation}
		\log \frac{f_-(x)}{f_+(b_0)} = \int\limits_{[b_0,x)} \frac{d\log(f)}{df}\,df = \int\limits_{[b_0,x)\setminus D_f} \frac{df}{f} + \int\limits_{D_f} \frac{\log f_+(x) - \log f_-(x)}{f_+(x)-f_-(x)}\,df
	\end{equation}
	and hence
	\begin{equation} \label{eq:log3.5}
		\int\limits_{[b_0,x)\setminus D_f} \frac{df}{f} = \log \frac{f_-(x)}{f_+(b_0)} + \sum_{x_n\in D_f} \log \frac{f_-(x_n)}{f_+(x_n)}.
	\end{equation}
	One can see from the following inequality
	\begin{equation}
		0 < \log \frac{f_-(x_n)}{f_+(x_n)} \le \frac{f_-(x_n) - f_+(x_n)}{f_+(x_n)} = \frac{a_n}{1-a_n}
	\end{equation}
	that the series
	\begin{equation}
		\sum_{n=1}^{\infty} \log \frac{f_-(x_n)}{f_+(x_n)}
	\end{equation}
	converges provided the series $\sum_{1}^{\infty} a_n$ converges. Hence we obtain that the integral on the left hand side of~\eqref{eq:log3.5} diverges which completes the proof.
\end{proof}

\begin{definition}[\cite{Ever66,Ever73,Ever76}] \label{def:SLP_D}
	Let the system $S[R_1,R_2]$ be singular at $b$. It is said to be in the \emph{strong limit point} case if
		\begin{equation} \label{eq:def_SLP}
			\lim_{x\to b} u_1(x)v_2(x) = 0\quad\text{for any }\quad
			(u_1,u_2, f),\ (v_1,v_2, g) \in \cT;
		\end{equation}
		and it is said to have \emph{the Dirichlet property} if
		\begin{equation} \label{eq:def_D}
			\int_0^b |u_2(t)|^2dR_1(t) <\infty\quad\text{ for any } \quad (u_1,u_2, f)\in\cT.
		\end{equation}
\end{definition}

\begin{theorem} \label{thm:SLP}
	Let the system $S[R_1,R_2]$ be singular at $b$. Then the following statements are equivalent:
	\begin{enumerate}
		\item[(LP)]   The system $S[R_1,R_2]$ is in the limit point case.
		\item[(D)]    The system $S[R_1,R_2]$ has the Dirichlet property.
		\item[(SLP$^*$)] For any $(u_1,u_2, f) \in \cT$ the following equality holds
		\begin{equation} \label{eq:def_SLP_var}
			\lim_{x\to b} u_1(x)u_2(x) = 0.
		\end{equation}
		\item[(SLP)]  The system $S[R_1,R_2]$ is in the strong limit point case.
	\end{enumerate}
\end{theorem}
\begin{proof}

Without loss of generality we assume here that the functions $u_1$, $u_2$, and $f$ are real-valued. By the first Green's identity~\eqref{eq:G1} one obtains
\begin{equation*}
	\int_0^x u_2^2\, dR_1 = \int_0^x f u_1\, dR_2 + u_1 u_2|_0^x,
\end{equation*}
and hence
\begin{equation}
	\lim_{x\to b} u_1(x) u_2(x) = d
\end{equation}
where $d \in \mathbb{R}$ if the Dirichlet property holds and $d=+\infty$ otherwise.

Let us start with the implication (LP) $\Rightarrow$ (D).
For this purpose we assume the contrary i.e. the system $S[R_1,R_2]$ is in the limit point case but $d=+\infty$.
Notice, that the functions $R_1$ and $R_2$ do not have common points of discontinuity, therefore neither do the functions $u_1$ and $u_2$.
It implies that both $u_1$ and $u_2$ preserve their signs on some interval $[b_0,b)$ (otherwise they would have to share a jump from a positive to a negative value or vice versa), so the function $u_1$ is either positive and increasing or negative and decreasing.
If $1 \notin \cL^2(R_2)$ then it immediately results as $u_1 \notin \cL^2(R_2)$.
In the case if $1 \in \cL^2(R_2)$ (and hence $R_1 \notin \cL^2(R_2)$) the implication $f \in \cL^2(R_2) \Rightarrow f \in \cL^1(R_2)$ is valid and hence (see~\eqref{eq:f2d}) there exists a finite limit $u_2(b) \coloneqq \lim_{x\to b} u_2(x)$.
The limit $u_2(b)$ must be zero, otherwise from
\begin{equation}
	|u_1(x) - u_1(b_0)| = \left| \int_{b_0}^x u_2\,dR_1 \right| \ge \frac{|u_2(b)|}{2} (R_1(x)-R_1(b_0))
\end{equation}
one gets $u_1 \notin \cL^2(R_2)$.
One can see that $1/u_2 \notin \cL^2(R_2)$.
Indeed, if $1/u_2 \in \cL^2(R_2)$ then the integral
\begin{equation}
	\int_0^x \frac{f}{u_2} dR_2 = - \int_0^x \frac{du_2}{u_2}
\end{equation}
converges as $x\to b$, which contradicts to Lemma~\ref{lem:log}.
Since $d=+\infty$, the estimate $|u_1| > 1/|u_2|$ hold on some interval $[b_0,b)$ and provides again $u_1 \notin \cL^2(R_2)$.
This completes the proof of the implication (LP) $\Rightarrow$ (D).

Now let us prove the implication (D) $\Rightarrow$ (SLP*). We first will show that $d=0$.
In the case $1 \in \cL^2(R_2)$ the reasoning of the previous paragraph can be used to show that $u_1 \notin \cL^2(R_2)$ for every non-zero $d$.
In the case $1 \notin \cL^2(R_2)$ the reasoning above shows again that $u_1 \notin \cL^2(R_2)$ for every $d>0$.
Therefore we assume $d<0$ and get that $u_1$ is either positive and decreasing or negative and increasing on some interval $[b_0,b)$, namely $u_1 \to 0$ as $x\to b$. From $|u_1 u_2| > |d|/2$ on $[b_0,b)$ (with a possible change of point $b_0$) we obtain the following inequality
\begin{equation}
	\int_{b_0}^b u_2^2\,dR_1 = \int_{b_0}^b u_2\,df_1 >
	\frac{d}{2} \int_{b_0}^b \frac{du_1}{u_1} = +\infty.
\end{equation}
The left hand side converges by our assumption but the right hand side diverges due to Lemma~\ref{lem:log}. This contradiction proves that $d=0$. Thus, implication (D) $\Rightarrow$ (SLP) is valid.

As is known (see \cite[Theorem 4.3]{Str19}), the system $S[R_1,R_2]$ is in the limit point case if and only if for every $(u_1,u_2, f)$ and $(v_1,v_2, g)$ from $\cT$
\begin{equation} \label{eq:LP-W_crit}
	\lim_{x\to b} [u,v]_x = \lim_{x\to b} (u_1(x)v_2(x) - u_2(x)v_1(x)) = 0.
\end{equation}

In order to prove the implication (SLP*) $\Rightarrow$ (SLP) we notice first that by Lemma~\ref{lem:psi} the system $S[R_1,R_2]$ cannot be in the limit circle case since~\eqref{eq:def_SLP_var} holds for every $(u_1,u_2, f) \in \cT$. The condition~\eqref{eq:def_SLP} follows from~\eqref{eq:def_SLP_var}, \eqref{eq:LP-W_crit} and the following equality (cf.~\cite{Ever76})
\begin{equation}
	2 u_1(x)v_2(x) = (u_1 + v_1)(u_2 + v_2) + [u,v]_x = 0.
\end{equation}

Assume that the statement (SLP) holds, i.e. condition~\eqref{eq:def_SLP} is satisfied for every $(u_1,u_2, f)$ and $(v_1,v_2, g)$ from $\cT$. Then, clearly, \eqref{eq:LP-W_crit} holds for every $(u_1,u_2, f)$ and $(v_1,v_2, g)$ from $\cT$ and hence the system $S[R_1,R_2]$ is in the limit point case. This proves the implication (SLP) $\Rightarrow$ (LP).
\end{proof}

\begin{remark}
In the case of absolutely continuous $R_1$ and $R_2$ the implication $(LP)\Rightarrow(SLP)$ for the system $S[R_1,R_2]$ was proved in~\cite{Kalf74},
see also~\cite{Ever76}.
\end{remark}
\subsection{Boundary triples for integral systems in the limit point case}
\begin{definition}\label{def:Neumann_m}
Let the system $S[R_1,R_2]$ be in the limit point case at $b$.
Then for each $\lambda\in\nCR$ there is a unique coefficient $m_N(\lambda)$, such that
\begin{equation} \label{eTWeyl_1}
	\psi_1(\cdot,\lambda) = s_1(\cdot,\lambda) - m_N(\lambda) c_1 (\cdot,\lambda)\in \cL^2(R_2).
\end{equation}
The function $m_N$
is called the \emph{Neumann $m$-function} of the system \eqref{eq:IntSys} on $I$
and the function $\psi(t,\lambda)$ is called the Weyl solution of the system $S[R_1,R_2]$ on $I$.
\end{definition}

Let us collect some statements concerning boundary triples for $S^*$, which were partially formulated in~\cite{Str18,Str19}.
\begin{proposition}\label{prop:BT_IS_LP}
	Let the system $S[R_1,R_2]$ be in the limit point case at $b$, and let $T=T_{min}$. Then:
	\begin{enumerate}
		\item [(i)] $T$ is a symmetric nonnegative operator in $L^2(R_2)$ with deficiency indices $(1,1)$.
		\item [(ii)] The triple $\Pi = (\dC,\Gamma_0,\Gamma_1)$, where
		\begin{equation}\label{eq:BTLP}
			\Gamma_0 \bm{u} = u_2(0), \quad \Gamma_1 \bm{u} = -u_1(0), \quad \bm{u} \in T^*,
		\end{equation}
		is a boundary triple for $T^*$.
		\item[(iii)] The defect subspace $\sN_\lambda(T)$ is spanned by the Weyl solution $\psi_1(t,\lambda)$, and the Weyl function $m(\lambda)$ of $T$ corresponding to the boundary triple $\Pi$ coincides with the Neumann $m$-function of the system $S[R_1,R_2]$ on $I$:
		\begin{equation}\label{eq:WFLP}
			m(\lambda)=- \frac{\psi_1(0,\lambda)}{\psi_2(0,\lambda)} = m_N(\lambda).
		\end{equation}
		\item[(iv)] The Weyl function $m(\lambda)$ of $T$ corresponding to the boundary triple $\Pi$ coincides with the principal Titchmarsh-Weyl coefficient $q(\lambda)$ of  the system $S[R_1,R_2]$ on $I$ and belongs to the Stieltjes class $\cS$.
		\item[(v)] If $R_2(b)<\infty$ then the Weyl function $m_N$ of $T_N$ admits the representation
		\begin{equation}\label{eq:Pol_mLP}
			m_N(\lambda) = -\frac{1}{R_2(b)\cdot\lambda} + \wt m(\lambda);
		\end{equation}
		where $\wt m$ is a function from $\cS$ such that $\lim_{y\downarrow 0} y \wt m(iy) = 0$.
	\end{enumerate}
\end{proposition}
\begin{proof}
\textbf{1.} At first we show ${\textrm (i)-(ii)}$.
Since \eqref{eq:IntSys} is in the limit point case at $b$,
\begin{equation*}
	\lim_{x\to b} [u,\conj{v}]_x = 0 \quad \text{for} \quad \bm{u} = \begin{bmatrix} \pi u_1 \\ \pi f \end{bmatrix},\ \bm{v} = \begin{bmatrix} \pi v_1 \\ \pi g \end{bmatrix}\in T_{\max}
\end{equation*}
and hence
the generalized Green's identity \eqref{eq:G2} is of the form
\begin{equation}\label{eq:GreenLP}
  \int_0^b (f\conj{v_1} - u_1\conj{g})\, dR_2(t) = -[u,\conj{v}]_0 = u_2(0)\conj{v_1(0)} - u_1(0)\conj{v_2(0)}.
\end{equation}
Therefore, the triple $\Pi$ in~\eqref{eq:BTLP} is a boundary triple for $T^*$.

It follows from the first Green's identity \eqref{eq:G1} and Lemma~\ref{lem:psi} that for every $\bm{u} \in T$ the identity~\eqref{eq:Green1_} holds and thus the linear relation $T$ is nonnegative.

\textbf{2.} Now ${\textrm (iii)}$ is shown.
In the limit point case there is only one linearly independent solution  $\psi(\cdot,\lambda)$ of the system $S[R_1,R_2]$ such that  $\psi_1(\cdot,\lambda)\in L^2(R_2)$, see~\eqref{eTWeyl_1},  and hence the defect subspace $\sN_\lambda(T^*)$ is spanned by the function $\psi_1(\cdot,\lambda)$.
Denote $\bm{u}(t,\lambda)=(\psi_1(\cdot,\lambda),\lambda\psi_1(\cdot,\lambda))^T\in\wh \sN_\lambda(T^*)$.
It follows from~\eqref{eq:BTLP} that
\begin{equation*}
	\Gamma_0\bm{u}(\cdot,\lambda) =  \psi_2(0,\lambda) = 1,\quad
	\Gamma_1\bm{u}(\cdot,\lambda) = -\psi_1(0,\lambda) = m_N(\lambda),
\end{equation*}
This yields formula~\eqref{eq:WFLP}.

\textbf{3.} Now we show ${\textrm (iv)}$.
If $\lambda\in\dR_-$ then it follows from Lemma~\ref{lem:2.4} that the function $\frac{s_1(x,\lambda)}{c_1(x,\lambda)}$ is increasing and  bounded from above. Therefore, the following limit
\begin{equation}\label{eq:mLP}
  q(\lambda) \coloneqq \lim_{x\to b}\frac{s_1(x,\lambda)}{c_1(x,\lambda)}
\end{equation}
exists and is nonnegative for every $\lambda\in\dR_-$. By Stieltjes-Vitaly theorem the function $q$ is holomorphic on $\dC\setminus[0,\infty)$. The function $q$ belongs to the Stieltjes class $\cS$, since it is nonnegative for every $\lambda\in\dR_-$. Since  $\frac{s_1(x,\lambda)}{c_1(x,\lambda)}$ belongs to the Weyl disc $D_x(\lambda)$ and the system $S[R_1,R_2]$ is limit point  at $b$, for every $\lambda\in\dC_+\cup\dC_-$ the following equality holds
\begin{equation}
  q(\lambda)=\lim_{x\to b}\frac{s_1(x,\lambda)}{c_1(x,\lambda)}=m_N(\lambda).
\end{equation}

\textbf{4.} Assume that $R_2(b)<+\infty$. Let us consider the family of von Neumann $m$-functions $m_N^x(\lambda)=\frac{s_2(x,\lambda)}{c_2(x,\lambda)}$
converging to $m_N(\lambda)$ as $x\to b-$.
Due to equality~\eqref{eq:cs2}
\begin{equation}\label{eq:c2s2}
 \frac{1}{ m_N^x(\lambda)}= \frac{c_2(x,\lambda)}{s_2(x,\lambda)}=\int_0^x\frac{-\lambda}{s_2(x,\lambda) s_{2+}(x,\lambda)}\,dR_2(x).
\end{equation}
Since $s_2(x,\lambda)\ge 1$ for $x\in[0,b)$ and $\lambda\in\dR_-$ there exists the limit
\begin{equation*}
\frac{-1}{\lambda m_N(\lambda)}=\lim_{x\to b}\frac{-c_2(x,\lambda)}{\lambda s_2(x,\lambda)}=\int_0^b\frac{1}{s_2(x,\lambda) s_{2+}(x,\lambda)}\,dR_2(x)
\end{equation*}
Due to Lemma~\ref{lem:2.4}
\begin{equation*}
	\lim_{\lambda\downarrow 0}\frac{1}{s_2(x,\lambda) s_{2+}(x,\lambda)} = 1,
	\quad \text{and} \quad
	\left|\frac{1}{s_2(x,\lambda) s_{2+}(x,\lambda)}\right|\le 1
	\quad \text{for} \quad x\in[a,b).
\end{equation*}
Hence one obtains by the Lebesgue bounded convergence Theorem
\begin{equation}
	\lim_{\lambda\to 0} \frac{1}{-\lambda m_N(\lambda)} = \int_{[0,b)}dR_2 = R_2(b).
\end{equation}
This implies (v).

\end{proof}

\subsection{The canonical singular continuation of a regular integral system}
If the integral system $S[R_1,R_2]$ is regular at $b$ then due to Remark~\ref{rem:Param} we can assume without loss of generality that $b<\infty$.
\begin{definition}
	For a regular system $S[R_1,R_2]$ with $b<\infty$ we define the extended functions
	\begin{equation}\label{eq:SingCont}
		\wt R_1(x) \coloneqq \left\{
		\begin{array}{cc}
	    	R_1(x) :& x\in [0,b], \\
	    	R_1(b) :& x\in [b,\infty),
	    \end{array}\right. \quad
		\wt R_2(x) \coloneqq \left\{
		\begin{array}{cc}
			R_2(x) :& x\in [0,b], \\
			R_2(b)+x-b :& x\in [b,\infty).
		\end{array}\right.
	\end{equation}
	The integral system $S[\wt R_1,\wt R_2]$ corresponding to \begin{equation} \label{eq:ContIntSys}
		\wt u(x,\lambda) = \wt u(0,\lambda)+\int_0^x
		\begin{bmatrix}
			0 & d\wt{R}_1(t) \\
			-\lambda d\wt{R}_2(t) & 0
		\end{bmatrix} \wt u(t,\lambda),\quad x\in [0, \infty)
	\end{equation}
	will be called \emph{the canonical singular continuation} of a regular integral system $S[R_1,R_2]$.
\end{definition}

\begin{proposition}\label{prop:SingCont}
Let the  integral system $S[R_1,R_2]$, see~\eqref{eq:IntSys}, be regular at $b<\infty$.
Then the principal Titchmarsh-Weyl coefficient $\wt q$ of its canonical singular continuation
$S[\wt R_1,\wt R_2]$
coincides with the principal Titchmarsh-Weyl coefficient $q$ of the system $S[R_1,R_2]$:
\begin{equation}\label{eq:wt_q}
	\wt q(\lambda) = q(\lambda), \quad \lambda\in\dC\setminus\dR.
\end{equation}
\end{proposition}
\begin{proof}
Let the pair $u_1,u_2$ satisfy the integral system $S[R_1,R_2]$ for some $\lambda\in\dC\setminus\dR$
and let  $\wt u_1,\wt u_2$ be the continuations of $u_1,u_2$ to the interval $[0,+\infty)$ given by
\begin{equation}\label{eq:wt_uu}
	\left\{
		\begin{aligned}
			\wt u_1(x,\lambda)&=u_1(b,\lambda),\quad x\in(b,\infty), \\
			\wt u_2(x,\lambda)&=u_2(b,\lambda)-\lambda u_1(b,\lambda)(x-b),\quad x\in(b,\infty).
		\end{aligned}
	\right.
\end{equation}
Then the pair $\wt u_1,\wt u_2$ satisfies the integral system~\eqref{eq:ContIntSys}.
If $c_1,c_2$ and $s_1,s_2$ are solutions of ~\eqref{eq:IntSys} according to the initial conditions~\eqref{eq:csK} then the continuations $\wt c_1,\wt c_2$ and $\wt s_1,\wt s_2$ are
solutions of the integral system~\eqref{eq:ContIntSys} with the same initial conditions~\eqref{eq:csK}.

In view of~\eqref{eq:wt_uu} the principal Titchmarsh-Weyl coefficient $\wt q$ of the canonical singular continuation
$S[\wt R_1,\wt R_2]$ is of the form
\begin{equation*}
	\wt q(\lambda) =
	\lim_{x\to\infty}\frac{\wt s_1(x,\lambda)}{\wt c_1(x,\lambda)} =
	\lim_{x\to\infty}\frac{s_1(x,\lambda)}{c_1(x,\lambda)} =
	q(\lambda).
\end{equation*}
\end{proof}

\section{Dual integral systems}
\label{sec:dual_strings}
\begin{definition}\label{def:dual_systems}
	\emph{The dual system} $\dual$ to a singular system $S[R_1,R_2]$ is defined by changing the roles of $R_1$ and $R_2$ in~\eqref{eq:IntSys}, that is $\hat S[R_1,R_2] = S[R_2,R_1]$ and
\begin{equation} \label{eq:DualIntSys}
	\wh u(x,\lambda) = \wh u(0,\lambda) + \int_0^x
	\begin{bmatrix}
		0 & dR_1(t) \\
		-\lambda dR_2(t) & 0
	\end{bmatrix} \wh u(t,\lambda), \quad x \in [0,b).
\end{equation}
In case the system $S[R_1,R_2]$ is regular, we will denote by $\dual$ the dual to its canonical singular continuation:
 $\hat S[R_1,R_2]=S[\wt R_2, \wt R_1]$.
\end{definition}

Let $\wh{s}(\cdot,\lambda)$ and $\wh{c}(\cdot,\lambda)$ be the unique solutions of \eqref{eq:DualIntSys} satisfying the initial conditions
\begin{equation}
	\wh{c}_1(0,\lambda) = 1, \ \wh c_2(0,\lambda) = 0,
	\quad \text{and} \quad
	\wh{s}_1(0,\lambda) = 0, \ \wh s_2(0,\lambda) = 1.
\end{equation}

\begin{theorem} \label{thm:2}
Let $U(x,\lambda)$ and $\wh U(x,\lambda)$ be the fundamental matrices of the
system $S[R_1,R_2]$ and its dual system $\dual$ respectively.
Let $m_N$ and $\wh m_N$ be the Neumann $m$-functions of the
systems $S[R_1,R_2]$ and $\dual$ in the sense of Definitions~\ref{def:5.1}, \ref{def:Neumann_m}.
Then:
 \begin{enumerate}
   \item [(i)] The matrices $U(x,\lambda)$ and $\wh U(x,\lambda)$ are related by
\begin{equation}\label{eq:Fund_dir_dual}
	\wh U(x,\lambda) = D(\lambda)^{-1} U(x,\lambda) D(\lambda),
	\quad \text{where} \quad
	D(\lambda) =
	\begin{pmatrix}
		0 & -\lambda^{-1} \\
		1 & 0
	\end{pmatrix}.
\end{equation}
\item [(ii)] If the system $S[R_1,R_2]$ is singular at $b$,
 then
\begin{equation}\label{eq:whm_mLP}
  \wh m_N(\lambda)=-\frac{1}{\lambda m_{N}(\lambda )}.
\end{equation}

\item [(iii)] If $S[R_1,R_2]$ is regular at $b$, then
\begin{equation}\label{eq:whm_m}
  \wh m_N(\lambda )=\frac{\wh s_2(b,\lambda )}{\wh c_2(b,\lambda )}=-\frac{c_1(b,\lambda )}{\lambda  s_1(b,\lambda )}
  =  -\frac{1}{\lambda m_{ND}(\lambda )},
\end{equation}
where $m_{ND}(\lambda )$ is the Neumann $m$-function of system $S[R_1,R_2]$, subject to the boundary condition $u_1(b)=0$, see Definition~\ref{def:5.1D}.

\item [(iv)] The principal Titchmarsh-Weyl coefficients $q$ and $\wh q$ of $S[R_1,R_2]$ and $\dual$ are connected by the equality
\begin{equation}\label{eq:mS_whmS2}
	\wh q(\lambda) = -\frac{1}{\lambda q(\lambda)},
	\quad \lambda\in\dC\setminus\dR_+.
\end{equation}
\end{enumerate}
\end{theorem}
\begin{proof}
\textbf{1.} At first (i) is shown.
A straightforward calculation shows that the solutions $\wh{s}(\cdot,\lambda)$ and
$\wh{c}(\cdot,\lambda)$ of \eqref{eq:DualIntSys} are related to the
solutions ${s}(\cdot,\lambda)$ and
${c}(\cdot,\lambda)$ of \eqref{eq:IntSys} by the equalities
\begin{equation}\label{eq:wh_cs}
  \begin{bmatrix}
    \wh{c}_1(\cdot,\lambda) \\
    \wh{c}_2(\cdot,\lambda)
  \end{bmatrix}
=  \begin{bmatrix}
    {s}_2(\cdot,\lambda) \\
    -\lambda{s}_1(\cdot,\lambda)
  \end{bmatrix},
  \quad
  \begin{bmatrix}
    \wh{s}_1(\cdot,\lambda) \\
    \wh{s}_2(\cdot,\lambda)
  \end{bmatrix}
  =  \begin{bmatrix}
   -\lambda^{-1} {c}_2(\cdot,\lambda) \\
    {c}_1(\cdot,\lambda)
  \end{bmatrix}.
\end{equation}
The equality~\eqref{eq:Fund_dir_dual} follows from \eqref{eq:wh_cs}.

System $S[R_1,R_2]$ is regular at $b$ if and only if both $S[R_1,R_2]$ are in the limit circle case at $b$. Therefore the proof of (ii) can be splitted into the following three cases \textbf{2--4}.

\textbf{2.} \emph{Both $S[R_1,R_2]$ and $\dual$ are in the limit point case at $b$$:$}\\
Let $m_N$ be the Neumann $m$-function of the
systems $S[R_1,R_2]$, see Definition~\ref{def:Neumann_m},
and let $\psi_1(\cdot,\lambda)$ be the corresponding Weyl solution of the system $S[R_1,R_2]$.
Then the vector function
\begin{equation}\label{eq:wh_psi}
  \wh\psi(\cdot,\lambda) \coloneqq
    \begin{bmatrix}
    \wh{s}_1(\cdot,\lambda ) \\
    \wh{s}_2(\cdot,\lambda )
  \end{bmatrix}
  +\frac{1}{\lambda m_N(\lambda )}
      \begin{bmatrix}
    \wh{c}_1(\cdot,\lambda ) \\
    \wh{c}_2(\cdot,\lambda )
  \end{bmatrix}
  =  \begin{bmatrix}
   -\lambda^{-1} {c}_2(\cdot,\lambda )+ \lambda^{-1} m_N(\lambda)^{-1} s_2(\cdot,\lambda ) \\
    {c}_1(\cdot,\lambda ) - m_N(\lambda)^{-1} s_1(\cdot,\lambda )
  \end{bmatrix}
\end{equation}
is a solution of the system \eqref{eq:DualIntSys}. Moreover, due to Lemma~\ref{thm:SLP} $\wh \psi_{1}(\cdot,\lambda)= \frac{1}{\lambda m_N(\lambda)} \psi_{2}(\cdot,\lambda)$ belongs to $L^2(R_1)$.
Therefore, $\wh\psi_1(\cdot,\lambda)$ is the Weyl solution of the system $\dual$ and the function
$-\frac{1}{\lambda m_N(\lambda)}$ is the Neumann $m$-function of the system~$\dual$.

\textbf{3.} \emph{$S[R_1,R_2]$ is in the limit circle case and $\dual$ is in the limit point case at $b$$:$}\\
Let the function $\psi^N$ be defined by~\eqref{eTWeyl2}.
Since \eqref{eq:IntSys} is in the limit circle case it follows from Lemma~\ref{lem:psi} that $\psi^N_2\in L^2(R_1)$.
Hence,
$\wh\psi(\cdot,\lambda)$ is a solution of the system $\dual$, such that $\wh \psi_{1}(\cdot,\lambda)=\frac{1}{\lambda m_N(\lambda )}\psi^N_{2}(\cdot,\lambda)\in L^2(R_1)$.
Therefore, $\wh \psi_1$ is the {\em Weyl solution} of the system $\dual$ and the function
$-\frac{1}{\lambda m_N(\lambda )}$ is the Neumann $m$-function of the systems $\dual$.

\textbf{4.} \emph{$S[R_1,R_2]$ is in the limit point case and $\dual$ is in the limit circle case at $b$$:$}\\
As was shown on Step \textbf{3} the Neumann $m$-function $\wh m_{N}(\lambda )$ of the
systems $\dual$ subject to the boundary condition $\wh\psi_{2}(b,\lambda)=0$ is connected with
the Neumann $m$-function $m_{N}(\lambda )$ of the system $S[R_1,R_2]$ by the equality
\begin{equation*}
	m_N(\lambda) = -\frac{1}{\lambda \wh m_{N}(\lambda)}
\end{equation*}
which is equivalent to~\eqref{eq:whm_mLP}.

\textbf{5.} Now (iii) is shown.
Let $m_{ND}(\lambda )$ be the Neumann $m$-function of the system $S[R_1,R_2]$, subject to the boundary condition~\eqref{eq:m_functLCD} and let $\psi_1^{ND}(\cdot,\lambda )$ be the corresponding {\em Weyl solution} of the system $S[R_1,R_2]$ defined by~\eqref{eTWeylND}. By definition $\psi^{ND}_{1}(b,\lambda )=0$.
Then the vector function
\begin{equation*}
	\wh\psi(\cdot,\lambda) \coloneqq
		\begin{bmatrix}
		\wh{s}_1(\cdot,\lambda) \\
		\wh{s}_2(\cdot,\lambda)
	\end{bmatrix}
	+\frac{1}{\lambda m_{ND}(\lambda)}
	\begin{bmatrix}
		\wh{c}_1(\cdot,\lambda) \\
		\wh{c}_2(\cdot,\lambda)
	\end{bmatrix}
	= -\frac{1}{m_{ND}(\lambda)}
	\begin{bmatrix}
		-\frac{1} {\lambda }\left(s_2(\cdot,\lambda)- m_{ND}(\lambda) c_2(\cdot,\lambda)\right) \\
		s_1(\cdot,\lambda)- m_{ND}(\lambda) c_1(\cdot,\lambda)
	\end{bmatrix}
\end{equation*}
is a solution of the system \eqref{eq:DualIntSys} such that $\wh \psi_{2}(b,\lambda)=\psi^{ND}_{1}(b,\lambda)=0$.
 Therefore, the function $\frac{-1}{\lambda m_{ND}(\lambda )}$ is the Neumann $m$-function of the
systems $\dual$, subject to the boundary condition $\wh \psi_{2}(b,\lambda)=0$.

\textbf{6.} Finally (iv) is shown.
If the integral system $S[R_1,R_2]$ is singular at $b$ then the Neumann $m$-function $m_N$ (resp. $\wh m_N$)
coincides with the principal Titchmarsh-Weyl coefficient $q$ of the system~$S[R_1,R_2]$
(resp. $\wh q$ of the system~ $\dual$), see Propositions~\ref{prop:BT_IS_N}, \ref{prop:BT_IS_LP}.
Therefore, \eqref{eq:mS_whmS2} is implied by \eqref{eq:whm_mLP}.

If the system $S[R_1,R_2]$ is regular at $b$ then by Propositions~\ref{prop:SingCont} $q$ coincides with
the principal Titchmarsh-Weyl coefficient $\wt q$ of the canonical singular continuation
$S[\wt R_1,\wt R_2]$ of the system $S[R_1,R_2]$ to $[0,+\infty)$, see~\eqref{eq:SingCont}.
By the statement of the above paragraph the principal Titchmarsh-Weyl coefficient $\wh q$ of the dual system $S[\wt R_2,\wt R_1]$ is of the form
\begin{equation*}
	\wh q(\lambda) = -\frac{1}{\lambda \wt q(\lambda)} =
	-\frac{1}{\lambda q(\lambda )},
\end{equation*}
and \eqref{eq:mS_whmS2} is shown.
 \end{proof}

Since the relation of duality for integral systems is reflexive one derives from the proof of Theorem~\ref{thm:2} the following statement.
\begin{corollary}\label{cor:psi}
Let the system $S[R_1,R_2]$ be in the limit point case and let $\dual$ be in the limit circle case at $b$.
Let $\psi_1(\cdot,\lambda)$ be the corresponding Weyl solution of the system $S[R_1,R_2]$.
Then
\begin{equation}\label{eq:psi_0}
	\lim_{x\to b} \psi_1(x,\lambda) = 0.
\end{equation}
\end{corollary}

\begin{proof}
	As it was mentioned in the proof of Theorem~\ref{thm:2} (Step \textbf{3}), the Weyl solution $\psi(\cdot,\lambda)$ of the system $S[R_1,R_2]$ is connected with the Weyl solution $\wh\psi^N(\cdot,\lambda)$ of the dual system \eqref{eq:IntSys} by the equality $\psi_{1}(\cdot,\lambda)=\frac{1}{\lambda\wh m_N(\lambda )}\wh\psi^N_{2}(\cdot,\lambda)$.
	Since $\wh\psi^N_{2}(b,\lambda)=0$ one obtains~\eqref{eq:psi_0}.
\end{proof}
\begin{remark}
	Formula~\eqref{eq:whm_mLP} was proven in \cite{KWW06} for Krein strings and in \cite{Kost13} for integral systems.
	However, in \cite{Kost13} it was overlooked that in the regular case the formula \eqref{eq:whm_mLP} fails to hold and should be replaced by~\eqref{eq:whm_m}.
\end{remark}


\begin{thebibliography}{99}
\bibitem{Arens}
	R.~Arens, \emph{Operational calculus of linear relations}, Pacific J.
	Math., \textbf{11} (1961), 9--23.

\bibitem{ArDy12}
	D.\,Z.~Arov,  H.~Dym, \textit{Bitangential Direct and Inverse Problems for Systems of Integral and Differential Equations}, Cambridge University Press, Cambridge, 2012.

\bibitem{Atk64}
	F.~Atkinson, \emph{Discrete and continuous boundary problems},
	{Mathematics in Science and Engineering, Vol. 8},
	{Academic Press, New York-London}, {1964}, {xiv+570 p}.

\bibitem{Ben72}
    C.~Bennewitz, \emph{Symmetric relations on a Hilbert space}, Lect. Notes
    Math., \textbf{280} (1972), 212--218.

\bibitem{Ben89}
    C.~Bennewitz, \emph{Spectral asymptotics for Sturm-Liouville equations},
    Proc. London Math. Soc., \textbf{59} (1989), 294--338.

\bibitem{BerUsShe}
	Y.\,M. Berezansky, Z.\,G. Sheftel, G.\,F. Us, \emph{Functional analysis}, Vol. I. Operator Theory: Advances and Applications, 85. Birkhäuser Verlag, Basel, 1996. xx+423 pp.

\bibitem{DM91}
	V.\,A.~Derkach, M.\,M.~Malamud, \emph{Generalized resolvents and the boundary value problems for hermitian operators with gaps}, J. Funct. Anal. \textbf{95} (1991), 1--95.

\bibitem{DM95}
	V.\,A.~Derkach and M.\,M.~Malamud, \emph{The extension theory of hermitian operators and the moment problem}, J. Math. Sciences, \textbf{73} (1995), 141--242.

\bibitem{Ever66}
    W.\,N.~Everitt,
    \emph{On the limit point classification of second-order differential operators}, J. Lond. Math. Soc. \textbf{41} (1966), 531--544.

\bibitem{Ever73}
    W.\,N.~Everitt and M.~Giertz,
    \emph{A Dirichlet type result for ordinary differential operators},
    Math. Ann. \textbf{203} (1973), 119--128.

\bibitem{Ever76}
    W.\,N.~Everitt, \emph{A note on the Dirichlet condition for second-order differential expressions},
    Gan. J. Math. \textbf{28} (1976), 312--320.

\bibitem{EvEver}
    W.\,D.~Evans and W.\,N.~Everitt, A\emph{ return to the Hardy--Littlewood integral inequality}, Proc. Roy. Soc. Lond. A \textbf{380} (1982), 447--486.

\bibitem{Fel57}
    W.~Feller, \emph{Generalized second order differential operators and their lateral conditions}, Illinois J. Math., \textbf{1} (1957), 459--504.

\bibitem{Fel59}
    W.~Feller, \emph{The birth and death processes as diffusion processes}, J. Math. Pures Appl., \textbf{38} (1959), 301--345.

\bibitem{GG91}
    V.\,I.~Gorbachuk,  M.\,L.~Gorbachuk, \textit{Boundary value problems for
    operator differential equations}, Kluwer Academic Publishers Group,
    1991.

\bibitem{IMcK65}
    K.~Ito and H.\,P.~McKean Jr., \emph{Diffusion processes and their sample paths}, Springer, Berlin-Heidelberg-New York, 1965.

\bibitem{KacKr58}
	I.\,S.~Kac, M.\,G.~Kre\u{\i}n, \emph{Criteria for the discreteness of the spectrum of a singular string}, (Russian) Izv.\ Vysš.\ Učebn.\ Zaved.\ Matematika, \textbf{2} (1958), 136--153.

\bibitem{KaKr74}
    I.\,S.~Kac, M.\,G.~Kre\u{\i}n, \emph{$R$-functions--analytic functions mapping the upper halfplane into itself},
    Supplement I to the Russian edition of F.\,V.~Atkinson,
    \emph{Discrete and continuous boundary problems}, Mir, Moscow, 1968 (Russian).
    English translation: Amer. Math. Soc. Transl. Ser. (2)\textbf{103} (1974), 1--18.

\bibitem{KacK68}
    I.\,S. Kac and M.\,G. Krein, \emph{On the spectral functions of the string}, Supplement II to the Russian edition of F.\,V.~Atkinson,
    \emph{Discrete and continuous boundary problems}, Mir, Moscow, 1968 (Russian).
    English translation: Amer. Math. Soc. Transl., (2) \textbf{103} (1974), 19--102.

\bibitem{Kalf74}
	H.~Kalf, \emph{Remarks on some Dirichlet type results for semibounded Sturm--Liouville operators},
	Math. Ann. \textbf{210} (1974), 197--205.

\bibitem{KWW06}
	H.~Kaltenback, H.~Winker and H.~Woracek, \emph{Symmetric relations of finite negativity. Operator theory in Krein spaces and nonlinear eigenvalue problems}, 191--210, Oper. Theory Adv. Appl., \textbf{162}, Birkhauser, Basel, 2006.
\bibitem{KWW07}
	M.~Kaltenb\"ack, H.~Winkler and H.~Woracek,
	\textit{Strings, dual strings, and related canonical systems}, Math. Nachr. \textbf{280} (2007), no. 13--14, 1518--1536.

\bibitem{Kas75}
	Y.~Kasahara, \emph{Spectral theory of generalized second order differential operators and its applications to Markov processes}, Japan. J. Math. (N.S.) \emph{1} (1975/76), 67--84.

\bibitem{Koc75}
	A.\,N.~Kochubei, \emph{On extentions of symmetric operators and symmetric binary relations}, Matem. Zametki, \textbf{17}, no. 1 (1975), 41--48.

\bibitem{Kost13}
    A.~Kostenko, \emph{The similarity problem for indefinite Sturm-Liouville
        operators and the HELP inequality},
        Adv. Math. \textbf{246} (2013), 368--413.

\bibitem{LaSc90}
	H.~Langer, W.~Schenk,
	\emph{Generalized Second- Order Differential Operators, Corresponding Gap Diffusions and Superharmonic Transformations},
	Math. Nachr. \textbf{148} (1990), 746.

\bibitem{LM03}
	M.~Lesch, M.~Malamud,
	\emph{On the deficiency indices and self-adjointness of symmetric Hamiltonian systems},
	J. Differential Equations, \textbf{189} (2003), no. 2, 556--615.

\bibitem{M92}
	M.~Malamud,	\emph{On the formula of generalized resolvents of a nondensely defined Hermitian operator},
	{Ukr. Mat. Zh.}, \textbf{44} (1992), no.~12, 1658--1688.

\bibitem{Man68}
	{P.~Mandl},	\emph{Analytical Treatment of One-dimensional Markov Processes}, Academia, Springer, 1968.

\bibitem{Str18}
	D.~Strelnikov, \emph{Boundary Triples for Integral Systems on Finite Intervals},
	Journal of Mathematical Sciences, \textbf{231} (2018), no.~1, 83--100.

\bibitem{Str19}
	D.~Strelnikov, \emph{Boundary triples for integral systems on the half-line},
	Methods Funct. Anal. Topology, \textbf{25} (2019), no.~3., 84--96.

\bibitem{Hew60}
	E.~Hewitt, \emph{Integration by Parts for Stieltjes Integrals}, The American Mathematical Monthly, \textbf{67} (1960), no.~5, 419--423.
\end{thebibliography}
\end{document}